\documentclass[a4paper]{amsart}

\usepackage{eucal}
\usepackage{amsmath}
\usepackage{amssymb}
\usepackage{pstricks}
\usepackage{pstricks-add}
\psset{nodesep=2pt,labelsep=1pt,arrowscale=1.2}
\usepackage{enumerate}
\usepackage[colorlinks=false, backref=true]{hyperref}

\theoremstyle{plain}
\newtheorem{theorem}{Theorem}[section]
\newtheorem{proposition}[theorem]{Proposition}
\newtheorem{lemma}[theorem]{Lemma}
\newtheorem{corollary}[theorem]{Corollary}
\newtheorem{example}[theorem]{Example}
\theoremstyle{definition}
\newtheorem{definition}[theorem]{Definition}
\theoremstyle{remark}
\newtheorem*{remark}{Remark}
\newtheorem{problem}[theorem]{Problem}

\newcommand{\injto}{\hookrightarrow}
\newcommand{\epito}{\twoheadrightarrow}
\newcommand{\Pol}{\operatorname{Pol}}
\newcommand{\End}{\operatorname{End}}
\newcommand{\Emb}{\operatorname{Emb}}
\newcommand{\Aut}{\operatorname{Aut}}
\newcommand{\Age}{\operatorname{Age}}

\newcommand{\rank}{{\operatorname{rank}}}
\newcommand{\cf}{\operatorname{cf}}
\newcommand{\scf}{\operatorname{scf}}

\newcommand{\restr}{\mathord{\upharpoonright}}

\newcommand{\bA}{\mathbf{A}}
\newcommand{\bB}{\mathbf{B}}
\newcommand{\bC}{\mathbf{C}}

\newcommand{\bF}{\mathbf{F}}
\newcommand{\bH}{\mathbf{H}}

\newcommand{\bT}{\mathbf{T}}
\newcommand{\bU}{\mathbf{U}}
\newcommand{\bV}{\mathbf{V}}
\newcommand{\bW}{\mathbf{W}}

\newcommand{\cA}{\mathcal{A}}
\newcommand{\cB}{\mathcal{B}}
\newcommand{\cC}{\mathcal{C}}

\newcommand{\cG}{{\mathcal{G}}}

\newcommand{\cU}{\mathcal{U}}
\newcommand{\uF}{\mathbb{F}}
\newcommand{\uS}{\mathbb{S}}

\newcommand{\ba}{\bar{a}}
\newcommand{\bb}{\bar{b}}
\newcommand{\bc}{\bar{c}}

\newcommand{\bQ}{\mathbb{Q}}
\newcommand{\bN}{\mathbb{N}}

\newcommand{\fA}{\mathfrak{A}}
\newcommand{\fB}{\mathfrak{B}}
\newcommand{\fC}{\mathfrak{C}}
\newcommand{\fD}{\mathfrak{D}}

\newcommand{\Ht}{\operatorname{Ht}}
\newcommand{\HE}{{\mathcal{E}}}

\newcommand{\Fraisse}{Fra\"\i{}ss\'e}

\title[Universal homomorphisms, universal structures,\dots]{Universal homomorphisms, universal structures, and the polymorphism clones of homogeneous structures}
\author{Christian Pech}
\address{Christian Pech\\Institut f\"ur Algebra\\Technische Universit\"at Dresden\\01062 Dresden\\Germany}
\email{Christian.Pech@tu-dresden.de}
\author{Maja Pech}
\address{Maja Pech\\Institut f\"ur Algebra\\Technische Universit\"at Dresden\\01062 Dresden\\Germany\\ and Department of Mathematics and Informatics\\University of Novi Sad\\Trg Dositeja Obradovi\'ca 4\\ 21000 Novi Sad\\Serbia}
\thanks{supported by the Ministry of Education and Science of the Republic of Serbia through Grant No.174018,  by the grant (Contract 114-451-1901/2011) of the the Secretariat of Science and Technological Development of the Autonomous Province of Vojvodina, and DAAD reinvitation scholarship.}
\email{Maja.Pech@tu-dresden.de, maja@dmi.uns.ac.rs}
\subjclass[2010]{08A05 (03C15,08C05,03C50,03C35)}
\keywords{universal structure, homogeneous structure, retract, polymorphism clone, cofinality, Bergman property $\aleph_0$-categoricity, oligomorphic permutation group}
\begin{document}

\begin{abstract}
Using a categorial version of \Fraisse's theorem due to Droste and G\"obel, we derive a criterion for a comma-category to have universal homogeneous objects. 

As a first application  we give  new existence result for universal structures and for $\aleph_0$-categorical universal structures. 

As a second application we  characterize the retracts of a large class of homogeneous structures, extending previous results by Bonato, Deli\'c, Dolinka, and Kubi\'s. 

As a third application we show for a large class of homogeneous structures that their polymorphism clone is generated by polymorphisms of bounded arity, generalizing  a classical result by Sierpi\'nski that the clone of all functions on a given set is generated by its binary part.    

Further we study the cofinality and the Bergman property for clones and we give sufficient conditions on a homogeneous structure to have a polymorphism clone that has uncountable cofinality and the Bergman property. 
\end{abstract}
\maketitle

\section{Introduction}

The mathematics in this paper is inspired by an idea of Richard Rado \cite{Rad64} to use universal functions in order to construct universal graphs and universal simplicial complexes. While the problem of proving or disproving the existence of universal objects in given classes of structures has been getting a considerable amount of attention in combinatorics (we just mention \cite{KomMekPac88,KomPac91,CheKom94,CheSheChi99,HubNes09}, but it is impossible to give a comprehensive overview of the relevant literature on this subject), it seems that  Rado's fundamental idea to study universal functions was never followed up upon --- perhaps because  it does not allow an obvious  generalization away from graphs and complexes towards a wider class of structures. Instead, for the construction of universal structures most often  the  theory of amalgamation classes and homogeneous structures due to \Fraisse{} and J\'onsson (\cite{Fra53,Jon56,Jon60}) is used, that does not have this handicap. 
 
On the first sight Rado's approach and the approach using \Fraisse-theory have little in common. However, if we change the perspective just slightly, this impression changes. From a categorial point of view, structures are objects and functions are morphisms of suitable categories. Moreover, morphisms can be objects of categories, too (of the so called comma-categories). So from this point of view the difference between functions and structures gets blurred. In \cite{DroGoe92} Droste and G\"obel developed the \Fraisse-J\'onsson theory for categories (this was generalized later on by Kubi\'s \cite{Kub07} to a wider class of categories). This allows to construct universal objects in a wide class of categories. The first result of our paper is a criterion that gives sufficient conditions for comma-categories to have  universal, homogeneous objects (cf. Theorem~\ref{mainconstruction}). In a sense this is the theoretical spine of this paper as it allows to construct universal morphisms which in turn are used in all subsequent results.

In Section~\ref{s4} the categorial results from Section~\ref{s6} are instantiated into the model-theoretic world. The notion of universal homogeneous homomorphisms is defined, and sufficient conditions for the existence of universal homomorphisms to a given structure within a given strict \Fraisse-class of structures are given (cf. Theorem~\ref{hom-constraints}). Moreover, Proposition~\ref{oligomorphic} gives sufficient conditions for the oligomorphy of the automorphism group of the domain of a universal homogeneous homomorphism. Universal homomorphisms are used to construct universal structures for several classes of model-theoretic structures. Here we understand the term universal in its strong sense --- that every element of the given class is isomorphic to an induced substructure of the universal structure. In particular we show:
\begin{enumerate}[(1)]
	\item\label{ex1} the class of countable wellfounded posets of bounded height contains a universal structure,
	\item\label{ex2} for every countable graph $\bH$ there is a universal countable $\bH$-colorable graph (a graph $\bA$ is called $\bH$-colorable if there exists a homomorphism from $\bA$ to $\bH$); moreover this structure can be chosen to be  $\aleph_0$-categorical whenever $\bH$ is finite or $\aleph_0$-categorical,
	\item\label{ex3} the class of all countable structures that are homomorphism equivalent with a given structure $\bH$ contains a universal structure; moreover it has an $\aleph_0$-categorical universal object if $\bH$ is finite or $\aleph_0$-categorical,
	\item\label{ex4} the class of all countable directed acyclic graphs contains an $\aleph_0$-categorical universal object.
\end{enumerate}
To our knowledge, for \eqref{ex1} and \eqref{ex3} the existence of a universal structure was not known before, while for \eqref{ex2} and \eqref{ex4} the existence of a universal object was known before (cf. \cite{HubNes09b}), but the existence of $\aleph_0$-categorical universal structures for these classes was not known. In this context it is noteworthy that till now all classes of countable structures in which the question about the existence of a universal object was settled, were elementary (i.e., axiomatizable in first order logics). In fact usually the considered classes are defined by a class of forbidden finite substructures. In contrast to this, e.g. the class of well-founded posets of bounded height can not be described by forbidding a set of finite substructures. Indeed, as it is not closed with respect to directed unions, this class is not elementary.  Thus our techniques allow the construction of universal structures for non-elementary classes of structures. 

In Section~\ref{s5} we demonstrate how the categorial theory of universal homogeneous morphisms developed in Section~\ref{s6} can be used to study the structure of the endomorphism monoids of countable homogeneous structures. The focus is on the classification of retracts of such structures (up to isomorphism). Our work extends previous work on this problem by Bonato, Deli\'c \cite{BonDel00}, Mudrinski \cite{Mud10}, Dolinka \cite{Dol12}, and Kubi\'s \cite{Kub13}. Of particular interest to us are  such retracts $\bB$ of a homogeneous structure $\bU$ that are induced by a universal homogeneous retraction $r:\bU\epito\bB$ (here universal homogeneous retractions are a special case of universal homogeneous homomorphisms introduced previously in Section~\ref{s4}). In Theorem~\ref{uhretract} for every countable homogeneous structure all such retracts are characterized. We further characterize all countable homogeneous structures with the property that  all retracts are induced by universal homogeneous retractions (cf. Corollary~\ref{allretsunivhom}). 

In the remainder of the paper we examine several aspects of the structure of the polymorphism clones of countable homogeneous structures. Clones on infinite sets have been studied by many authors including Goldstern, Pinsker, P\"oschel, Rosenberg, Shelah, and many more. Fields of application range through universal algebra, multivalued logic, model theory, set-theory, and theoretical computer science. For a survey on results about clones on infinite sets and for further references we refer to \cite{GolPin08}.

Section~\ref{genclones} is concerned with the problem of describing generating sets of polymorphism clones of countable homogeneous structures. A classical result by Sierpi\'nski states that the clone of all finitary functions on a set $A$ is generated by the set of binary functions on this set.  Our initial motivation was the question whether the polymorphism clone of the Rado graph has a generating set of bounded arity. The main result of this section is Theorem~\ref{clgen}. It identifies a class of countable homogeneous structures (including the Rado graph) whose polymorphism clone is generated by the binary polymorphisms (in fact it is shown that for these structures the polymorphism clone is generated by the homomorphic self embeddings together with one unary and one binary polymorphism). Several natural examples for such structures are given.

In Section~\ref{cofinality} we study an intriguing structural invariant of clones on infinite sets --- the cofinality. Roughly speaking, the cofinality of a clone is the smallest possible length of a chain of proper subclones that approximates a given clone. Cofinality-questions have previously been studied  for groups \cite{Ser74,KopTit74,Sab75,MacNeu90,Gou92,HodHodLasShe93,Tho97,DroGoe02,DroGoe05,DroTru09,DroTru11},  semigroups \cite{MalMitRus09,Dol11}, and other general algebraic structures \cite{Kop77,KopMcKMon84,GouMorTsi86}. If a clone has no generating set of bounded arity, then it follows that its cofinality is $\aleph_0$. Thus of interest for us are clones that have a generating set of bounded arity. Theorem~\ref{clcf} connects Theorem~\ref{clgen} with a result by Dolinka \cite{Dol11} to identify a class of homogeneous structures whose polymorphism clone has uncountable cofinality. It turns out that all examples given in Section~\ref{genclones} (including the Rado graph) have polymorphism clones with uncountable cofinality.

In \cite{Ber06} Bergman made the beautiful observation that the connected Cayley graphs of the full symmetric groups (over arbitrary sets) have finite diameter, and he also gave a short and  elegant proof of a result by Macpherson and Neumann, that the symmetric groups on infinite sets have uncountable cofinality. An infinite group whose connected Cayley graphs all have a finite diameter is said to have the Bergman property. In \cite{DroHol05}, the concept of strong cofinality is introduced (though this name for the property appears for the first time in \cite{DroGoe05}) and it is shown that a group has uncountable strong cofinality if and only if it has uncountable cofinality and the Bergman property. Meanwhile the concept of strong cofinality and the Bergman property have been studied thoroughly \cite{DroGoe05,DroHol05,Khe06,Tol06,Tol06b,Cor06,KecRos07,DroHolUlb08,Ros09}. In \cite{MalMitRus09} the Bergman property and the notion of strong cofinality are defined and studied for semigroups. It is shown that strong uncountable cofinality is the same as uncountable cofinality plus Bergman property for semigroups. Moreover, for many concrete semigroups the strong uncountable cofinality is shown. Further results about the Bergman property for semigroups can be found in \cite{MitPer11,Dol11}.  

In Section~\ref{sBergman} the Bergman property and the notion of strong cofinality are defined for clones. In Proposition~\ref{bergman} we show that a clone has uncountable strong cofinality if and only if it has uncountable cofinality and the Bergman property. The analogous result for groups was proved by Droste and Holland in \cite{DroHol05}, and for semigroups by Maltcev, Mitchell and Ru\v{s}kuc in\cite{MalMitRus09}. Finally, Theorem~\ref{newscf} characterizes a class of countable homogeneous  structures whose polymorphism clones have the Bergman property. Thereafter a number of examples (including the Rado graph) is given.

Before we start, let us finally fix some general notions and notations: 
This paper deals (mainly) with model theoretic structures. Whenever we do not state otherwise, a signature can contain any number of relational symbols, but only countably many functional and constant symbols. Structures will be denoted like $\bA,\bB,\bC,\dots$. There carriers are denoted like $A,B,C,\dots$. Tuples are denoted like $\ba,\bb,\bc$, and usually $a_i$ will denote the $i$-th coordinate of $\ba$, etc. For the basic model theoretic terminology we refer to \cite{Hod97}

For classes $\cA$, $\cB$ of structures we write $\cA\to\cB$ if for every $\bA\in\cA$ there exists a $\bB\in\cB$ and a homomorphism $f:\bA\to\bB$. Instead of $\{\bA\}\to\cB$ we write $\bA\to\cB$ and instead of $\cA\to \{\bB\}$ we write $\cA\to\bB$. An important special case is the notion $\bA\to\bB$ which means that there is a homomorphism from $\bA$ to $\bB$. If $\bA\to\bB$ and $\bB\to\bA$, then we call $\bA$ and $\bB$ \emph{homomorphism-equivalent}. 

If $\cA$ is a class of structures over a signature $R$, then by $(\cA,\to)$ and $(\cA,\injto)$ we will denote the categories of objects from $\cA$ with homomorphisms or embeddings as morphisms, respectively (note that throughout this paper we understand by an embedding a strong injective homomorphisms in the model-theoretic sense).

 The \emph{age} of a structure $\bA$ is the class of all finitely generated  structures that embed into $\bA$ (it is denoted by $\Age(\bA)$).

Let $\cC$ be a class of finitely generated  structures over the same signature. We say that $\cC$ has the \emph{Joint embedding property (JEP)} if whenever $\bA,\bB\in \cC$, then there exists a $\bC\in\cC$ such that both $\bA$ and $\bB$ are embeddable in $\bC$. Moreover, $\cC$ has the \emph{Hereditary property (HP)} if whenever $\bA\in \cC$, and $\bB$ is a finitely generated substructure of $\bA$,  then $\bB$ is isomorphic to an element of $\cC$. 

A basic theorem by Roland \Fraisse{} states that a class of finitely generated structures of the same type is the age of a countable structure if and only if it has only countably many isomorphism types, it has the (HP) and the (JEP). Therefore a class of finitely generated countable structures with these properties will be called an \emph{age}. If $\cC$ is an age, then by $\overline{\cC}$ we denote the class of all countable structures whose age is contained in $\cC$ (in \Fraisse's terminology, $\overline{\cC}$ contains all countable structures younger than $\cC$). 

Finally, recall that a  structure $\bU$ is called \emph{homogeneous} if every isomorphism between finitely generated substructures of $\bU$ extends to an automorphism of $\bU$. The standard examples of homogeneous structures are $(\bQ,\le)$, and the Rado graph (a.k.a. the countable random graph). If $\cC$ is the age of a countable homogeneous structure, then, besides the (HP) and the (JEP), $\cC$ has another property --- the amalgamation property. In general, an age $\cC$ is said to have the \emph{amalgamation property (AP)} if for all $\bA,\bB_1,\bB_2\in\cC$ and for all embeddings $f_1:\bA\injto\bB_1$ and $f_2:\bA\injto\bB_2$ there exists some $\bC\in\cC$ together with embeddings $g_1:\bB_1\injto\cC$ and $g_2:\bB_2\injto\bC$ such that the following diagram commutes:
	\[
		\begin{psmatrix}
		[name=A]\bA & [name=B1]\bB_1\\
		[name=B2]\bB_2 & [name=C]\bC.
		\ncline{H->}{A}{B1}^{f_1}%
		\ncline{H->}{A}{B2}<{f_2}%
		\ncline[linestyle=dashed]{H->}{B2}{C}^{g_2}%
		\ncline[linestyle=dashed]{H->}{B1}{C}<{g_1}%
		\end{psmatrix}
	\]
\begin{theorem}[\Fraisse{} (\cite{Fra53})]
  A class $\cC$ of finitely generated structures of the same type is the age of some
  countable homogeneous structure if and only if
  \begin{enumerate}[(i)]
  \item it has only countably many non-isomorphic members,
  \item it has the (HP), the (JEP), and the (AP).
  \end{enumerate}
  Moreover, any two countable homogeneous relational structures
  with the same age are isomorphic.\qed
\end{theorem}
If $\bU$ is a homogeneous structure with age $\cC$ then we call $\bU$ a \emph{\Fraisse-limit} of $\cC$. Moreover, ages that have the (AP) will  be called \emph{\Fraisse-classes}.

At last let us stress that that this paper is staged in ZFC. We did not mark the places where the axiom of choice is actually used. For us, cardinals are special ordinals (the initial ordinals). Moreover we use the von-Neumann-style definition of ordinals as hereditarily transitive sets. 
In particular, an ordinal number $\alpha$ is a set that consists of all ordinal numbers smaller than $\alpha$, and is well-ordered by the $\in$-relation.

\section{Universal homogeneous objects in comma categories}\label{s6}
The main result of this section will be a theorem that gives sufficient conditions for a comma-category to contain a universal homogeneous object. This result relies on a categorical version of \Fraisse's theorem due to Droste and G\"obel  \cite{DroGoe92}.

For the convenience of the reader, this part is kept relatively self-contained. For basic notions from category theory we refer to \cite{HCA1}.

Usually, if $h$ is a morphism from an object $A$ to an object $B$ in a given category $\fC$, then we will simply write $h:A\to B$ whenever the category $\fC$ is clear from the context. In cases when there is a danger of confusion, we will write $h\in\fC(A\to B)$, instead.

\subsection{Comma-categories}
Recall the definition of a comma category:
\begin{definition}
   Let  $\fA$,$\fB$,$\fC$ be categories, let $F:\fA\to\fC$, $G:\fB\to\fC$ be functors. The comma category $(F \downarrow G)$ has as objects triples $(A,f,B)$ where $A\in\fA$, $B\in\fB$, $f:FA\to GB$. The morphisms from $(A,f,B)$ to $(A',f',B')$ are pairs $(a,b)$  such that $a\in\fA(A\to A')$, $b\in\fB(B\to B')$ such that the following diagram commutes:
\[
	\begin{psmatrix}
 		FA & GB\\
		FA' & GB'
		\ncline{->}{1,1}{1,2}^{f}
		\ncline{->}{2,1}{2,2}^{f'}
		\ncline{->}{1,1}{2,1}<{Fa}
		\ncline{->}{1,2}{2,2}<{Gb}
	\end{psmatrix}
\]
\end{definition}
There are two projection functors $U:(F \downarrow G)\to \fA$ and $V:(F \downarrow G)\to \fB$ defined by $U:(A,f,B)\mapsto A, (a,b)\mapsto a$ and $V:(A,f,B)\mapsto B, (a,b)\mapsto b$. Moreover there is a canonical natural transformation $\alpha: F U\Rightarrow G V$ defined by $\alpha_{(A,f,B)}=f$ (cf. \cite[Prop.1.6.2]{HCA1}). Finally, the comma-categories have the following universal property:
\begin{proposition}[{\cite[Prop.1.6.3]{HCA1}}]\label{univcomma}
	With the notions from above, let $\fD$ be another category, let $U':\fD\to\fA$, $V':\fD\to\fB$ be functors, and let $\alpha':FU'\Rightarrow GV'$ be a natural transformation. Then there exists a unique functor $W:\fD\to(F \downarrow G)$ such that $UW=U'$, $VW=V'$, $\alpha*W=\alpha'$ (where $\alpha*W$ denotes the horizontal composition of $\alpha$ and $W$, i.e. $\alpha'_D=\alpha_{WD}$ for all $D\in\fD$).\qed
\end{proposition}

Let $\fC$ be a category, $\fD$ be a small category, and let $F:\fD\to\fC$ be a functor. Recall that a \emph{compatible cocone of $F$} is a pair $(S,(f_X)_{X\in\fD})$ where $S\in\cC$, $f_X:FX\to S$ such that for all $h\in\fD(X\to Y)$ the following diagram commutes:
\[
\begin{psmatrix}[colsep=1cm]
	& [name=S]S \\
	[name=FX]FX & & [name=FY]FY
	\ncline{->}{FX}{FY}^{Fh}
	\ncline{->}{FX}{S}<{f_X}
	\ncline{->}{FY}{S}>{f_Y}
\end{psmatrix}
\]
A compatible cocone $(S,(f_X)_{X\in\fD})$ for $F$ is called \emph{weakly limiting cocone} if for every other compatible cocone $(T,(g_X)_{X\in\fD})$ of $F$ there exists a morphism $m:S\to T$ such that for all $X\in\fD$ the following diagram commutes:
\[
	\begin{psmatrix}[colsep=1cm]
	[name=S]S & & [name=T]T\\
	& [name=FX]FX
	\ncline{->}{S}{T}^{m}
	\ncline{->}{FX}{S}<{f_X}
	\ncline{->}{FX}{T}>{g_X}
	\end{psmatrix}
\]
The morphism $m$ is called a \emph{mediating morphism}. The object $S$ is called a \emph{weak colimit of $F$}. 
If for every compatible cocone $(T,(g_X)_{X\in\fD})$ of $F$ there is a unique mediating morphism $m:S\to T$, then $(S,(f_X)_{X\in\fD})$ is called a \emph{limiting cocone of $F$}. In this case, $S$ is called a \emph{colimit of $F$}. 

The following lemmata show how arrow categories behave with respect to (weak) colimits. 
\begin{lemma}\label{weak}
	With the notions from above, let $\fD$ be a small category, and let $H:\fD\to (F \downarrow G)$. Suppose that
	\begin{enumerate}
		\item $U H$ has a weakly limiting cocone $(L,(p_D)_{D\in\fD})$,
		\item $V H$ has a compatible cocone $(M,(q_D)_{D\in\fD})$,
		\item $(FL,(Fp_D)_{D\in\fD})$ is a weakly limiting cocone of $FUH$.
	\end{enumerate}
	Then there is a morphism $h: FL\to GM$ such that $((L,h,M), (p_D,q_D)_{D\in\fD})$ is a compatible cocone for $H$. In case that $(FL,(Fp_D)_{D\in\fD})$ is a limiting cocone of $FUH$, $h$ is unique.
\end{lemma}
\begin{proof}
	Recall $\alpha:FU\Rightarrow GV$ with $\alpha_{(A,f,B)}=f$  is a natural transformation. Consider $\alpha':=\alpha * H:FUH\Rightarrow GVH$ given by $\alpha'_D=\alpha_{HD}$. Then $(GM,(Gq_D\circ\alpha'_D)_{D\in \fD})$ is a compatible cocone for $FUH$. Hence there is a morphism  $h: FL\to GM$ such that for every $D\in \fD$ we have $Gq_D\circ\alpha'_D=h\circ Fp_D$. So indeed, $(p_D,q_D):HD\to (L,h,M)$. In case that $(FL,(Fp_D)_{D\in\fD})$ is a limiting cocone of $FUH$, there is a unique such $h$ so that  $(p_D,q_D):HD\to(L,h,M)$ is a morphism. 
	
	Let $d\in\fD(D\to D')$. Then $HD=(UHD, \alpha'_D, VHD)$ and $HD'=(UHD',\alpha'_{D'},VHD')$. Moreover, $Hd=(UHd,VHd)$. So we have that the following diagram commutes:
	\[
		\begin{psmatrix}[rowsep=1cm]
			&[name=FUHD]FUHD & [name=GVHD]GVHD\\
			[name=FL]FL & & & [name=GM]GM\\
			&[name=FUHD']FUHD' & [name=GVHD']GVHD'\\[-1cm]
			\ncline{->}{FUHD}{GVHD}^{\alpha'_D}
			\ncline{->}{FUHD'}{GVHD'}^{\alpha'_{D'}}
			\ncline{->}{FUHD}{FUHD'}<{FUHd}
			\ncline{->}{GVHD}{GVHD'}<{GVHd}
			\ncline{->}{FUHD}{FL}^{Fp_D}
			\ncline{->}{FUHD'}{FL}_{Fp_{D'}}
			\ncline{->}{GVHD}{GM}^{Gq_D}
			\ncline{->}{GVHD'}{GM}_{Gq_{D'}}
			\ncarc[arcangle=-90]{->}{FL}{GM}^h
		\end{psmatrix}
	\]
	It follows that $(p_D,q_D)=(p_{D'},q_{D'})\circ Hd$.  This shows that $((L,h,M),(p_D,q_D)_{D\in\fD})$ is a compatible cocone for $H$.
\end{proof}

\begin{lemma}\label{colimits}
	With the notions from above, let $\fD$ be a small category, and let $H:\fD\to (F \downarrow G)$. Suppose that
	\begin{enumerate}
		\item $U H$ has a limiting cocone $(L,(p_D)_{D\in\fD})$,
		\item $V H$ has a limiting cocone $(M,(q_D)_{D\in\fD})$,
		\item $(FL,(Fp_D)_{D\in\fD})$ is a limiting cocone of $FUH$,
	\end{enumerate}
	Then there is a unique morphism $h:FL\to GM$ such that $((L,h,M),(p_D,q_D)_{D\in\fD})$ is a compatible cocone of $H$. Moreover, this is a limiting cocone for $H$.
\end{lemma}
\begin{proof}
	From Lemma~\ref{weak} it follows that there is a morphism $h:FL\to GM$ such that $((L,h,M),(p_D,q_D)_{D\in\fD})$ is a compatible cocone for $H$. The uniqueness of $h$ follows from Lemma~\ref{weak}, too. So it remains to show that $((L,h,M),(p_D,q_D)_{D\in\fD})$ is a limiting cocone for $H$. 
	
	Let $((L',h',M'), (p'_D,q'_D)_{D\in\fD})$ be another compatible cocone for $H$. Then $(L',(p'_D)_{D\in\fD})$ is a compatible cocone for $UH$, and $(M',(q'_D)_{D\in\fD})$ is a compatible cocone for $VH$. Hence, there are unique morphisms $r:L\to L'$ and $s:M\to M'$ such that $r\circ p_D=p'_D$ and $s\circ q_D=q'_D$, for all $D\in\fD$. We will show that $(r,s):(L,h,M)\to (L',h',M')$ is the unique mediating morphism. First we need to show that it is a morphism at all: For this, we use that $(FL,(Fp_D)_{D\in\fD})$ is a limiting cocone for $FUH$. Consider the following diagram:
	\[
		\begin{psmatrix}
		& [name=FL] FL & [name=GM] GM\\
		& [name=FL'] FL' & [name=GM'] GM'\\
		[name=FUHD]FUHD & & & [name=GVHD] GVHD
		\ncline{->}{FL}{GM}^h
		\ncline{->}{FL'}{GM'}^{h'}
		\ncline{->}{FUHD}{GVHD}^{\alpha_D}
		\ncline{->}{FL}{FL'}<{Fr}
		\ncline{->}{GM}{GM'}>{Gs}
		\ncline{->}{FUHD}{FL'}^{Fp'_D}
		\ncline{->}{GVHD}{GM'}^{Gq'_D}
		\ncarc[arcangle=30]{->}{FUHD}{FL}<{Fp_D}
		\ncarc[arcangle=-30]{->}{GVHD}{GM}>{Gq_D}
		\end{psmatrix}
	\]
	The lower quadrangle commutes, because $(p'_D,q'_D)$ is a morphism. We already saw that the two triangles commute. Note that  $(GM', (Gs\circ h\circ Fp_D)_{D\in\fD})$ is a compatible cocone for $FUH$ with the mediating morphism $Gs\circ h$. Now we compute
	\begin{align*}
		h'\circ Fr\circ Fp_D &= h'\circ Fp'_D\\
		&= Gq'_D\circ\alpha_D \\
		&= Gs\circ Gq_D\circ\alpha_D\\
		&= Gs\circ h \circ Fp_D.
	\end{align*}
	Hence, $h'\circ Fr$ is another mediating morphism and we conclude that $h'\circ Fr= Gs\circ h$ and hence $(r,s)$ is a morphism. Since the two triangles of the above given diagram commute, it follows that $(r,s)$ is mediating. Let us show the uniqueness of $(r,s)$:
	
	Suppose that $(r',s')$ is another mediating morphism. Then $U(r',s')=r'$ is a mediating morphism between $(L, (p_D)_{D\in\fD})$ and $(L', (p'_D)_{D\in\fD})$, and $V(r',s')=s'$ is a mediating morphism between $(M, (q_D)_{D\in\fD})$ and $(M', (q'_D)_{D\in\fD})$. Hence $r=r'$ and $s=s'$.
\end{proof}

The following easy corollary is going to be useful later on when  from a colimit of $H:\fD\to(F\downarrow G)$ the colimits of $UH$ and $VH$ have to be deduced. 
\begin{corollary}\label{colimcor}
	With the notions from above, let $\fD$ be a small category, and let $H:\fD\to (F \downarrow G)$. Suppose that
	\begin{enumerate}
		\item $H$ has a limiting cocone $((L,h,M),(p_D,q_D)_{D\in\fD})$
		\item $U H$ has a limiting cocone $(L',(p'_D)_{D\in\fD})$,
		\item $V H$ has a limiting cocone,
		\item $(FL',(Fp'_D)_{D\in\fD})$ is a limiting cocone of $FUH$,
	\end{enumerate}
	Then $(L,(p_D)_{D\in\fD})$ is a limiting cocone for $UH$ and $(M,(q_D)_{D\in\fD})$ is a limiting cocone for $VH$.
\end{corollary}
\begin{proof}
	Suppose, that $(M',(q'_D)_{D\in\fD})$ is a limiting cocone for $VH$. Then, by Lemma~\ref{colimits}, there exists a unique morphism $h':FL'\to GM'$ such that $\big((L',h',M'),(p'_D,q'_D)_{D\in\fD}\big)$ is a limiting cocone for $H$. Hence, there exists a unique isomorphism $(p,q):(L',h',M')\to(L,h,M)$ that mediates between the two limiting cocones. I.e., for every $D\in\fD$, we have $(p,q)\circ(p_i',q_i')=(p_i,q_i)$. In particular, $p$ and $q$ are isomorphisms, and for all $D\in\fD$ we have $p\circ p'_i=p_i$ and $q\circ q'_i=q_i$. 
	Hence $(L,(p_d)_{D\in\fD})$ is a limiting cocone for $UH$ and $(M,(q_D)_{D\in\fD})$ is a limiting cocone for $VH$.
\end{proof}

\subsection{Algebroidal categories}
The notion of algebroidal categories goes back to  Banaschewski and Herrlich \cite{BanHer76}. We need this concept in order to be able  to make use of the category-theoretic version of \Fraisse's theorem due to Droste and G\"obel \cite{DroGoe92}. We closely follow the exposition from   \cite{DroGoe92}.

Let $\lambda$ be a regular cardinal. Let us consider $\lambda$ as a category (the objects are all ordinals $i<\lambda$, and there is a unique morphism from $i$ to $j$ whenever $i\le j$). 

A \emph{$\lambda$-chain} in $\fC$ is a functor from $\lambda$ to $\fC$. An object $A$ of $\fC$ is called \emph{$\lambda$-small} if whenever $(S,(f_i)_{i<\lambda})$ is a  limiting cocone of a $\lambda$-chain $F$ in $\fC$, and $h:A\to S$, then there exists a $j<\lambda$, and a morphism $g:A\to F(j)$ such that $h=f_j\circ g$. With $\fC_{<\lambda}$ we will denote the full subcategory of $\fC$ whose objects are all $\lambda$-small objects of $\fC$. The category $\fC$ will be called \emph{semi-$\lambda$-algebroidal} if for every $\mu\le\lambda$, all $\mu$-chains in $\fC_{<\lambda}$ have a colimit in $\fC$, and if every object of $\fC$ is the colimit of a $\lambda$-chain in $\fC_{<\lambda}$. Moreover, $\fC$ will be called \emph{$\lambda$-algebroidal} if 
\begin{enumerate}
	\item it is semi-$\lambda$-algebroidal, 
	\item $\fC_{<\lambda}$ contains at most $\lambda$ isomorphism classes of objects, and 
	\item between any two objects from $\fC_{<\lambda}$ there are at most $\lambda$ morphisms. 
\end{enumerate}
The following example describes the most paradigmatic examples of $\lambda$-algebroidal categories in this paper:
\begin{example}
	Let $\bA$ be a countable structure. Consider the category $\fA:=(\overline{\Age(\bA)},\injto)$ that has as objects all countable structures $\bB$ whose age is contained in $\Age(\bA)$ and as morphisms all embeddings between these structures. Then $\fA$ is $\aleph_0$-algebroidal and $\fA_{<\aleph_0}$ is the full subcategory of $\fA$ that is induced by $\Age(\bA)$. This follows directly from the observation that $\overline{\Age(\bA)}$ is closed with respect to directed unions of structures from $\Age(\bA)$, and that every structure is the union of its finitely generated substructures. Moreover between any two finitely generated structures there exist at most $\aleph_0$ homomorphism as every homomorphism is uniquely determined by its restriction to a generating set. 
\end{example}

The next example describes a class of $\lambda$-algebroidal categories that are in some sense degenerate. However, they will play an important role later on:
\begin{example}
	Consider a category $\fA$ that has just one object $T$ and such that all morphisms are isomorphisms (i.e. $\fA$ is essentially a group) morphisms. We claim that then $\fA$ is semi-$\lambda$-algebroidal for every regular cardinal $\lambda$. 
	
	Let us first show that $T$ is $\lambda$-small. Let $(T,(t_i)_{i<\lambda})$ be a limiting cocone of a $\lambda$-chain $H$ in $\fA$. Let $h:T\to T$. Let $i<\lambda$. Define $h':= t_i^{-1}\circ h$. Then $h= t_i\circ h'$.
	
	Now let us show that $\fA$ has colimits of $\mu$-chains for all $\mu\le\lambda$. Let $H$ be a $\mu$-chain in $\fA$. We define $t_i:H_i\to T$ by $t_i:=H(0,i)^{-1}$, for all $i<\mu$. We claim that $(T,(t_i)_{i<\mu})$ is a limiting cocone for $H$. First we show that it is compatible: Let $i<j<\mu$. Then, since $H(0,j)=(H(i,j)\circ H(0,i)$, we have 
	\[t_j=H(0,j)^{-1} = H(0,i)^{-1}\circ H(i,j)^{-1}= t_i\circ H(i,j)^{-1}.\]
	Hence we can compute
	\[ t_j\circ H(i,j) = t_i\circ H(i,j)^{-1}\circ H(i,j) = t_i.\]
	Let now $(T,(s_i)_{i<\mu}$ be another compatible cocone for $H$. We claim that $s_0:t\to T$ is a mediating cocone from $(T,(t_i)_{i<\mu})$ to $(T,(s_i)_{i<\mu})$. Indeed, for $i<\mu$ we compute
	\[ s_0\circ t_i = s_0\circ H(0,i)^{-1} = s_i\circ H(0,i)\circ H(0,i)^{-1}= s_i.\]
	Thus, indeed, $\fA$ has colimits of $\mu$-chains. 
	
	If we consider the trivial $\lambda$-chain $H$ that is defined by $Hi=T$ (for all $i<\lambda$) and $H(i,j)=1_T$ (for all $i<j<\lambda$). Then we see that $T$ is a colimit of $H$. Thus $\fA$ is semi-$\lambda$-algebroidal.
	
	If $\fA$ has at most $\lambda$ morphisms, then we even have that $\fA$ is $\lambda$-algebroidal. 	 
\end{example}

\begin{lemma}\label{chainext}
	Let $\fC$ be a semi-$\lambda$-algebroidal category all of whose morphisms are monos, and, for some   $\mu<\lambda$, let $F:\mu\to \fC$ be a $\mu$-chain such that for every $i<\mu$ we have $Fi\in\fC_{<\lambda}$. Let $S$ be a colimit of $F$ in $\fC$. Then $S\in\fC_{<\lambda}$.
\end{lemma}
\begin{proof}
	Let $(S,(s_i)_{i<\mu})$ be a limiting cocone of $F$. Let $H:\lambda\to\fC$ be a $\lambda$-chain with limiting cocone $(T,(t_j)_{j<\lambda})$. Finally, let $f:S\to T$. 
	
	For every $i<\mu$, since $Fi\in\fC_{<\lambda}$, there exists $j(i)<\lambda$ and $g_i:Fi\to Hj(i)$ such that $f\circ s_i=t_{j(i)}\circ g_i$. Let $j:=\bigcup_{i<\mu} j(i)$. Since $\lambda$ is a  regular cardinal, we have $j<\lambda$. 
	
	Now, for all $i<\mu$, define $\hat{s}_i:= H(j(i),j)\circ g_i$ (here $(j(i),j)$ is the morphism from $j(i)$ to $j$ in $\lambda$). We will show that $(Hj,(\hat{s}_i)_{i<\mu})$ is a compatible cocone for $F$: Let $i'>i$ such that $i'<\mu$. We compute
	\begin{align}\label{star}
		t_j\circ \hat{s}_i &= t_j\circ H(j(i),j)\circ g_i = t_{j(i)}\circ g_i = f\circ s_i\\
		&= f\circ s_{i'}\circ F(i,i') = t_{j(i')}\circ g_{i'}\circ F(i,i')\notag\\
		& = t_j\circ H(j(i'),j)\circ g_{i'}\circ F(i,i')=t_j\circ \hat{s}_{i'}\circ F(i,i').\notag
	\end{align}  
	Since $t_j$ is a mono, we conclude that $\hat{s}_i = \hat{s}_{i'}\circ F(i,i')$. Thus indeed $(Hj,(\hat{s}_i)_{i<\mu})$ is a compatible cocone for $F$. It follows that there is a mediating morphism $\hat{f}:S\to Hj$ between the limiting cocone $(S,(s_i)_{i<\mu})$ and the compatible cocone $(Hj,(\hat{s}_i)_{i<\mu})$. In other words, for all $i<\mu$ we have $\hat{f}\circ s_i = \hat{s}_i$. 
	
	Observe that $(T,(t_j\circ\hat{s}_i)_{i<\mu})$ is another compatible cocone of $F$ and that $t_j\circ\hat{f}$ is a mediating morphism between $(S,(s_i)_{i<\mu})$ and $(T,(t_j\circ\hat{s}_i)_{i<\mu})$. On the other hand, in \eqref{star}, we already computed that $f\circ s_i=t_j\circ\hat{s}_i$, for all $i<\mu$. Hence $f$ is another mediating morphism between the mentioned cocones. Since $(S,(s_i)_{i<\mu})$ is limiting, it follows that $f=t_j\circ \hat{f}$. This finishes the proof that $S$ is $\lambda$-small. 
\end{proof}
\begin{definition}
	Let now $\fA$, $\fC$ be categories, and let $\lambda$ be a cardinal. Then we call a functor $F:\fA\to\fC$ \emph{$\lambda$-cocontinuous} if whenever $(S,(s_i)_{i<\lambda})$ is a limiting cocone of a $\lambda$-chain $H$ in $\fA$, then $(FS,(Fs_i)_{i<\lambda})$ is a limiting cocone of $F H$.

	If, in addition, $\lambda$ is regular, $\fA$ is semi-$\lambda$-algebroidal, and if $\mu\le\lambda$. Then we call $F:\fA\to\fC$ \emph{$\mu_{<\lambda}$-cocontinuous} if whenever $(S,(s_i)_{i<\mu})$ is a limiting cocone of a $\mu$-chain $H$ of $\lambda$-small objects in $\fA$, then $(FS,(Fs_i)_{i<\mu})$ is a limiting cocone of $F H$.
\end{definition}

Let us  have a look onto $\lambda$-small objects in comma-categories. 
\begin{lemma}\label{smallthings}
	Let $\fA$, $\fB$, $\fC$ be categories,  such that $\fA$ and $\fB$ have colimits of  $\lambda$-chains and such that all morphisms of $\fB$ are monomorphisms. Let $F:\fA\to\fC$ be $\lambda$-cocontinuous, and let $G:\fB\to\fC$ any functor that preserves monomorphisms. 
	
	Let  $(A,f,B)$ be an object of $(F \downarrow G)$ such that  $A$ is $\lambda$-small in $\fA$ and $B$ is $\lambda$-small in $\fB$. Then $(A,f,B)$ is $\lambda$-small in $(F \downarrow G)$.
\end{lemma}
\begin{proof}
	Let $H:\lambda\to (F \downarrow G)$ be a $\lambda$-chain with limiting cocone $((L,h,M),(a_i,b_i)_{i<\lambda})$. Let $(a,b):(A,f,B)\to(L,h,M)$ be a morphism. Since $\fA$ and $\fB$ have colimits of $\lambda$-chains, it follows that $U H$, and $V H$ have colimits, and since $F$ is $\lambda$-cocontinuous, by Corollary~\ref{colimcor}, we have that $(L,(a_i)_{i<\lambda})$ is a limiting cocone of $U H$, and $(M,(b_i)_{i<\lambda})$ is a limiting cocone of $V H$. Since $A$ is $\lambda$-small in $\fA$ and $a:A\to L$, there exists an $i<\lambda$, and a morphism $\hat{a}:A\to UHi$ such that $a_i\circ\hat{a}=a$. Also, since $B$ is $\lambda$-small in $\fB$, and since $b:B\to M$, it follows that there exists some $j<\lambda$ and a morphism $\hat{b}:B\to VHj$ such that $b_j\circ\hat{b}=b$. Without loss of generality, $i=j$. Consider the following diagram:
	\[
		\begin{psmatrix}
			[name=FL]FL & & & [name=GM]GM\\
			& [name=FA]FA & [name=GB]GB\\
			[name=FUHi]FUHi & & & [name=GVHi]GVHi
			\ncline{->}{FL}{GM}^h
			\ncline{->}{FA}{GB}^f
			\ncline{->}{FUHi}{GVHi}^{\alpha_{Hi}}
			\ncline{->}{FA}{FUHi}^{F\hat{a}}
			\ncline{>->}{GB}{GVHi}^{G\hat{b}}
			\ncline{->}{FA}{FL}^{Fa}
			\ncline{>->}{GB}{GM}^{Gb}
			\ncline{->}{FUHi}{FL}<{Fa_i}
			\ncline{>->}{GVHi}{GM}<{Gb_i}
		\end{psmatrix}
	\]   
	where $Hi=(UHi,\alpha_{Hi},VHi)$. By the assumptions, the upper quadrangle and the two triangles of this diagram commute. We compute
	\begin{align*}
		Gb_i\circ \alpha_{Hi}\circ F\hat{a} &= h\circ Fa_i\circ 	F\hat{a} & \text{(since $ (a_i,b_i):Hi\to (L,h,M)$)}\\
		&=h\circ Fa\\
		&=Gb\circ f\\
		&=Gb_i\circ G\hat{b}\circ f
	\end{align*} 
	Since $Gb_i$ is a monomorphism, we conclude that $G\hat{b}\circ f=\alpha_{Hi}\circ F\hat{a}$ whence the whole diagram commutes. Consequently, $(\hat{a},\hat{b}):(A,f,B)\to Hi$ and $(a_i,b_i)\circ(\hat{a},\hat{b})=(a,b)$.
\end{proof}
Such $\lambda$-small objects $(A,f,B)$ in $(F \downarrow G)$ for which $A$ and $B$ are $\lambda$-small in $\fA$ and $\fB$, respectively, will be called \emph{inherited $\lambda$-small objects}.

\begin{lemma}\label{retractsmall}
Let $\fC$ be a semi-$\lambda$-algebroidal category, let $A$ be a $\lambda$-small object of $\fC$. Moreover, let $B\in\fC$, $\epsilon:A\epito B$, $\iota:B\injto A$, such that $\epsilon\circ \iota = 1_B$ (i.e., $B$ is a retract of $A$). Then $B$ is $\lambda$-small, too.
\end{lemma}
\begin{proof}
	Let $H:\lambda\to\fC$ be a $\lambda$-chain with limiting cocone $(S,(s_i)_{i<\lambda})$. Let $f:B\to S$. Then $f\circ\epsilon:A\to S$. Hence, there exists an $i<\lambda$, and a $g:A\to Hi$ such that $f\circ\epsilon= s_i\circ g$. It follows that $f=f\circ\epsilon\circ\iota=s_i\circ g\circ\iota$. Hence $B$ is $\lambda$-small.
\end{proof}

\begin{lemma}\label{inherited}
	With the notions from above, if in $(F \downarrow G)$ every object is the colimit of a $\lambda$-chain of inherited $\lambda$-small objects, then every $\lambda$-small object of $(F \downarrow G)$ is inherited.
\end{lemma}
\begin{proof}
	Let $(A,f,B)$ be $\lambda$-small in $(F \downarrow G)$, let $H:\lambda\to (F \downarrow G)$ such that $Hi=(A_i,f_i,B_i)$ is inherited $\lambda$-small for all $i<\lambda$, and such that $((A,f,B),(p_i,q_i)_{i<\lambda})$ is a limiting cocone of $H$. Consider the identity morphism $(1_A,1_B)$ of $(A,f,B)$. Since $(A,f,B)$ is $\lambda$-small, there is some $i<\lambda$ and some $(a,b):(A,f,B)\to(A_i,f_i,B_i)$ such that $(1_A,1_B)=(p_i,q_i)\circ(a,b)$. In other words, $(A,f,B)$ is a retract of $(A_i,f_i,B_i)$. It follows that $A$ is a retract of $A_i$ and $B$ is a retract of $B_i$. Hence, by Lemma~\ref{retractsmall} we have that $A$ is $\lambda$-small in $\fA$, and $B$ is $\lambda$-small in $\fB$. It follows that $(A,f,B)$ is inherited.
\end{proof}

\begin{definition}
	Let $\fA,\fB,\fC$ be categories. Let $F:\fA\to\fC$, $G:\fB\to\fC$ be functors. We say that \emph{$F$ preserves $\lambda$-smallness}  with respect to $G$ if for all $H:\lambda\to\fB$ with limiting cocone $(B,(g_i)_{i<\lambda})$, for all $A\in\fA_{<\lambda}$, and for all $f:FA\to GB$ there exists some $j<\lambda$ and some $h:FA\to GHj$ such that $Gg_j\circ h = f$.
\end{definition}

Recall that for ordinals $\mu\le\lambda$, we call a $\mu$-chain $J:\mu\to\lambda$ cofinal if
\[ \bigcup_{i<\mu} Ji = \lambda.\]
 
\begin{lemma}\label{subsequence}
	Let $\lambda$ be a regular cardinal, $\fA$ be a category, and let $H$ be a $\lambda$-chain in $\fA$ with limiting cocone $(S,(s_i)_{i<\lambda})$. Let $J:\lambda\to\lambda$ be cofinal. Then $(S,(s_{Ji})_{i<\lambda})$ is a limiting cocone for $HJ$.  
\end{lemma}
\begin{proof}
	Let $(T, (t_{Ji})_{i<\lambda})$ be a compatible cocone for $HJ$. We will complete this cocone to a compatible cocone for $H$. For an ordinal $k<\lambda$, by $\hat{k}$ we define the smallest ordinal greater or equal than $k$ that is in the image of $J$. This is well-defined, since $J$ is cofinal.  Then we put $t_k:= H(k,\hat{k})\circ t_{\hat{k}}$. To see that with $(T,(t_i)_{i<\lambda})$ we obtain a compatible cocone for $H$, let $k<l<\lambda$. 
	Then from one hand we have
	\[ t_l\circ H(k,l) = t_{\hat{l}}\circ H(l,\hat{l})\circ H(k,l) = t_{\hat{l}}\circ H(k,\hat{l}).\]
	On the other hand we may compute:
	\[ t_k= t_{\hat{k}}\circ H(k,\hat{k}) = t_{\hat{l}}\circ H(\hat{k},\hat{l})\circ H(k,\hat{k}) = t_{\hat{l}}\circ H(k,\hat{l})\]
	by the previous case. This proves that $(T,(t_i)_{i<\lambda})$ is indeed compatible for $H$. 
	
	Hence, there exists a morphism $h:S\to T$ that mediates between $(S,(s_i)_{i<\lambda})$ and $(T,(t_i)_{i<\lambda})$. I.e., for all $i<\lambda$ we have $h\circ s_i=t_i$. In particular, we have that $h$ mediates between $(S,(s_{Ji})_{i<\lambda})$ and $(T, (t_{Ji})_{i<\lambda})$. 
	
	Let $h':S\to T$ be another mediating morphism between  $(S,(s_{Ji})_{i<\lambda})$ and $(T, (t_{Ji})_{i<\lambda})$. Let $k<\lambda$. Then we compute
	\[ t_k = t_{\hat{k}}\circ H(k,\hat{k}) = h'\circ s_{\hat{k}}\circ H(k,\hat{k}) = h'\circ s_k.\]
	Hence $h'$ mediates also between $(S,(s_i)_{i<\lambda})$ and $(T,(t_i)_{i<\lambda})$. Thus we have $h=h'$ and the proof is finished. 
\end{proof}

\begin{proposition}\label{semialgebroidal}
	Let $\fA,\fB,\fC$ be categories such that $\fA$ and $\fB$ are semi-$\lambda$-algebroidal, and such that all morphisms of $\fA$ and $\fB$ are monomorphisms. Let $F:\fA\to\fC$, $G:\fB\to\fC$ be functors such that 
	\begin{enumerate}[(1)]
		\item $F$ is $\lambda$-cocontinuous and $\mu_{<\lambda}$-cocontinuous for all $\mu<\lambda$,
		\item $F$ preserves $\lambda$-smallness with respect to $G$, and 
		\item $G$ is $\lambda_{<\lambda}$-cocontinuous, 
		\item $G$ preserves monos.
	\end{enumerate}
	Then $(F \downarrow G)$ is semi-$\lambda$-algebroidal.
\end{proposition}
\begin{proof}
	Let us first show, that every object of $(F \downarrow G)$ is the colimit of a $\lambda$-chain of inherited $\lambda$-small objects: Let $(A,f,B)\in (F \downarrow G)$. Since $\fA$ is semi-$\lambda$-algebroidal, there is a $\lambda$-chain $H$ of $\lambda$-small objects in $\fA$ and morphisms $a_i\in\fA(Hi\to A)$ (for all $i<\lambda$) such that $(A,(a_i)_{i<\lambda})$ is a limiting cocone of $H$. Similarly, since $\fB$ is semi-$\lambda$-algebroidal, there is a $\lambda$-chain $K$  of $\lambda$-small objects in $\fB$ and a family of morphisms $b_i\in\fB(Ki\to B)$ ($i<\lambda$), such that $(B,(b_i)_{i<\lambda})$ is a limiting cocone of $K$. Since $G$ is $\lambda_{<\lambda}$-cocontinuous, we have that $(GB, (Gb_i)_{i<\lambda})$ is a limiting cocone of $GK$. Since $F$ preserves $\lambda$-smallness with respect to $G$, there exists a $j=j(i)$ and $h_i: FHi\to GKj(i)$ such that the following diagram commutes:
	\[
	\begin{psmatrix}
		[name=FA]FA & [name=GB]GB\\
		[name=FHi]FHi & [name=GKji]GKj(i)
		\ncline{->}{FA}{GB}^f%
		\ncline{->}{FHi}{GKji}_{h_i}%
		\ncline{->}{FHi}{FA}<{Fa_i}
		\ncline{>->}{GKji}{GB}>{Gb_{j(i)}}%
	\end{psmatrix}
	\]
	Whenever a factoring morphism of $f\circ Fa_i$ exists through $GKj$, then it exists also through $GKj'$ for all $j'>j$. Hence the function $J:\lambda\to\lambda: i\mapsto j(i)$ can be chosen to be increasing in a way that the sequence $(j(i))_{i<\lambda}$ is cofinal in $\lambda$. By taking $K':= KJ$, we have that $\chi:=(h_i)_{i<\lambda}$ is a natural transformation from $FH$ to $GK'$. Moreover, since $J$ is cofinal, by Lemma~\ref{subsequence} we have that $(B, (b_{j(i)})_{i<\lambda})$ is a limiting cocone of $K'$. By the universal property  of the comma-categories (cf. Proposition~\ref{univcomma}), there exists a unique functor $W:\lambda\to (F \downarrow G)$ such that $UW=H$, $VW=K'$, $\alpha*W=\chi$. It follows that $(A,f,B)$ is a colimit of $W$ and it follows from Lemma~\ref{smallthings} that $Wi$ is $\lambda$-small for all $i<\lambda$.

	It remains to show that $(F \downarrow G)$ has colimits of all $\mu$-chains of $\lambda$-small objects for all $\mu<\lambda$.
	
	As we showed above, every object in $(F\downarrow G)$ is the colimit of a $\lambda$-chain of inherited $\lambda$-small objects of  $(F\downarrow G)$. Hence, by Lemma~\ref{inherited} it follows that every object in $(F\downarrow G)_{<\lambda}$ is inherited. 
	
	Let now  $H:\mu\to (F\downarrow G)$ be a $\mu$-chain such that for all $i<\mu$ we have $Hi\in(F\downarrow G)_{<\lambda}$.  Then $UH$ is a $\mu$-chain of $\lambda$-small objects in $\fA$, and $VH$ is a chain of $\lambda$-small objects in $\fB$. Since $\fA$ and $\fB$ are both semi-$\lambda$-algebroidal,  it follows that $UH$ has a limiting cocone $(L,(p_i)_{i<\mu})$ and that $VH$ has a limiting cocone $(M,(q_i)_{i<\mu})$. Since $F$ is $\mu_{<\lambda}$-cocontinuous, we have that $(FL,(Fp_i)_{i<\mu})$ is a limiting cocone for $FUH$. Now, from Lemma~\ref{colimits} it follows that there exists a unique morphism $h:FL\to GM$ such that $((L,h,M),(p_i,q_i)_{i<\mu})$ is a limiting cocone for $H$. In particular, $(L,h,M)$ is a colimit of $H$. 
\end{proof}
In the proof of Proposition~\ref{semialgebroidal} we showed that  under the assumptions of Proposition~\ref{semialgebroidal}, all $\lambda$-small objects of $(F \downarrow G)$ are inherited. This enables us, to formulate the following result:
\begin{proposition}\label{algebroidal}
	Let $\fA,\fB,\fC$ be categories such that $\fA$ and $\fB$ are $\lambda$-algebroidal, and such that all morphisms of $\fA$ and of $\fB$ are monomorphisms. Let $F:\fA\to\fC$, $G:\fB\to\fC$ be functors such that 
	\begin{enumerate}[(1)]
		\item $F$ is $\lambda$-cocontinuous and $\mu_{<\lambda}$-cocontinuous for all $\mu<\lambda$,
		\item $F$ preserves $\lambda$-smallness with respect to $G$, and 
		\item $G$ is $\lambda_{<\lambda}$-cocontinuous, 
		\item $G$ preserves monos.
	\end{enumerate}
	Additionally, suppose that for all $\lambda$-small objects $A\in\fA_{<\lambda}$, $B\in\fB_{<\lambda}$ there are at most $\lambda$ morphisms between $FA$ and $GB$. Then $(F \downarrow G)$ is $\lambda$-algebroidal.
\end{proposition}
\begin{proof}
	By Proposition~\ref{semialgebroidal} we have that $(F \downarrow G)$ is semi-$\lambda$-algebroidal. We already noted, that all $\lambda$-small objects of $(F \downarrow G)$ are inherited. By this reason, the number of $\lambda$-small objects in $(F \downarrow G)$ is at most  $\lambda^3=\lambda$. Also, the number of morphisms between $\lambda$-small objects of $(F \downarrow G)$ is at most $\lambda^2=\lambda$. Hence, $(F \downarrow G)$ is $\lambda$-algebroidal. 
\end{proof}

\subsection{Universal homogeneous objects in categories}
In \cite{DroGoe92,DroGoe93}, Droste and G\"obel developed a categorial version of  classical model theoretic theorems by \Fraisse{} and J\'onsson that characterize universal homogeneous structures. This generalization is staged in $\lambda$-algebroidal categories, and we need to introduce a few more notions in order to be able to state it.

In the following, let $\fC$ be a category in which all morphisms are monomorphisms. Let $\fC^*$ be a full subcategory of $\fC$. 

Let $U\in\fC$. Then we say that 
\begin{description}
	\item[$U$ is $\fC^*$-universal] if for every $A\in\fC^*$ there is a morphism $f:A\to U$,
	\item[$U$ is $\fC^*$-homogeneous] if for every $A\in\fC^*$ and for all morphisms $f,g:A\to U$ there exists an automorphism $h$ of $U$ such that $h\circ f=g$,
	\item[$U$ is $\fC^*$-saturated] if for every $A,B\in\fC^*$ and for all $f:A\to U$, $g:A\to B$ there exists some $h:B\to U$ such that $h\circ g=f$.
\end{description}

We say that
\begin{description}
	\item[$\fC^*$ has the joint embedding property] if for all $A,B\in\fC^*$ there exists a $C\in\fC^*$ and morphisms $f:A\to C$ and $g:B\to C$,
	\item[$\fC^*$ has the amalgamation property] if for all $A$, $B$, $C$ from $\fC^*$ and $f:A\to B$, $g:A\to C$, there exists $D\in\fC^*$ and $\hat{f}:C\to D$, $\hat{g}:B\to D$ such that the following diagram commutes:
	\[
		\begin{psmatrix}
		[name=A]A & [name=B]B\\
		[name=C]C & [name=D]D.
		\ncline{->}{A}{B}^f%
		\ncline{->}{A}{C}<g%
		\ncline[linestyle=dashed]{->}{C}{D}^{\hat{f}}%
		\ncline[linestyle=dashed]{->}{B}{D}<{\hat{g}}%
		\end{psmatrix}
	\]
\end{description}

With these preparations done we are ready to  state  the result by Droste and G\"obel:
\begin{theorem}[{\cite[Thm.1.1]{DroGoe92}}]\label{DroGoe}
	Let $\lambda$ be a regular cardinal, and let $\fC$ be a $\lambda$-algebroidal category in which all morphisms are monomorphisms. Then, up to isomorphism, $\fC$ contains at most one $\fC$-universal, $\fC_{<\lambda}$-homogeneous object. Moreover, $\fC$ contains a $\fC$-universal, $\fC_{<\lambda}$-homogeneous object if and only if $\fC_{<\lambda}$ has the joint embedding property and the amalgamation property. 
\end{theorem}

Another result from \cite{DroGoe92} that we shall need is
\begin{proposition}[{\cite[Prop.2.2(a)]{DroGoe92}}]\label{univhomsat}
	Let $\lambda$ be a cardinal and let $\fC$ be a semi-$\lambda$-algebroidal category in which all morphisms are monic. Then for any object $U$ of $\fC$ the following are equivalent:
	\begin{enumerate}[(1)]
		\item $U$ is $\fC$-universal and $\fC_{<\lambda}$-homogeneous,
		\item $U$ is $\fC_{<\lambda}$-universal and $\fC_{<\lambda}$-homogeneous,
		\item $U$ is $\fC_{<\lambda}$-universal and $\fC_{<\lambda}$-saturated.
	\end{enumerate}
\end{proposition}

\begin{remark}
	Let $\cC$ be an age. We already noted in Example~\ref{ex1} that then $\fC:=(\overline{\cC},\injto)$ is $\aleph_0$-algebroidal, and that $\fC_{<\aleph_0}$ is the subcategory of $\fC$ induced by the elements of $\cC$. Let $\bU$ be an object of $\fC$. Then $\bU$ is $\fC$-universal if and only if it is universal for $\overline{\cC}$, it is $\fC_{<\aleph_0}$-homogeneous if and only if it is homogeneous, and it is $\fC_{<\aleph_0}$-saturated if and only if it is weakly homogeneous --- i.e. whenever $\bA$ and $\bB$ are finitely generated substructures of $\bU$ with $\bA\le\bB$, and $f:\bA\injto\bU$ is an embedding, then there exists an embedding $g:\bB\injto\bU$ that extends $f$.  
\end{remark}

\subsection{A \Fraisse-type theorem for comma-categories}\label{xs6}
Before we can come to the formulation of a sufficient condition that the comma-category of two functors has a universal homogeneous object, we need to introduce some more  notions.

\begin{definition}
	Let $\fA,\fB,\fC$ be categories, let $F:\fA\to\fC$, $G:\fB\to\fC$. 
	We say that $\fA$ has the $(F,G)$-joint embedding property if for all $B_1,B_2\in\fA$, $T\in\fB$, $h_1:FB_1\to GT$, $h_2:FB_2\to GT$ there exists a $C\in\fA$, $T'\in\fB$, $f_1:B_1\to C$, $f_2:B_2\to C$, $h:FC\to GT'$, $k_1,k_2:T\to T'$ such that the following diagram commutes:
	\[
	\begin{psmatrix}
[name=FB1] FB_1 & [name=FC] FC & [name=FB2] FB_2\\
[name=GT1] GT & [name=GTp]GT' & [name=GT2] GT%
		\ncline[linestyle=dashed]{->}{FB1}{FC}_{Ff_1}%
		\ncline[linestyle=dashed]{->}{FB2}{FC}_{Ff_2}%
		\ncline{->}{FB1}{GT1}<{h_1}%
		\ncline{->}{FB2}{GT2}>{h_2}%
		\ncline[linestyle=dashed]{->}{GT1}{GTp}^{Gk_1}%
		\ncline[linestyle=dashed]{->}{GT2}{GTp}^{Gk_2}%
		\ncline[linestyle=dashed]{->}{FC}{GTp}<{h}%
	\end{psmatrix}
	\]
	
	We say that $\fA$ has the \emph{$(F,G)$-amalgamation property} if for all $A,B_1,B_2\in\fA$, $f_1:A\to B_1$, $f_2:A\to B_2$ and for all $T\in \fB$, $h_1:FB_1\to GT$, $h_2:FB_2\to GT$ if $h_1\circ Ff_1=h_2\circ Ff_2$, then there exist $C\in\fA$, $g_1:B_1\to C$, $g_2:B_2\to C$, $T'\in\fB$, $k:T\to T'$, $h: FC\to GT'$ such that the following diagrams commute:
	\[
	\begin{psmatrix}
	& & &  & & & [name=GTp] GT'\\
	& & &  & & [name=GT] GT\\
	[name=B1] B_1 & [name=C] C & &  [name=FB1] FB_1 & [name=FC] FC\\
	[name=A] A & [name=B2]B_2 & & [name=FA] FA & [name=FB2] FB_2
	\ncline{->}{A}{B1}<{f_1}
	\ncline{->}{A}{B2}^{f_2}
	\ncline[linestyle=dashed]{->}{B1}{C}^{g_1}
	\ncline[linestyle=dashed]{->}{B2}{C}>{g_2}
	\ncline{->}{FA}{FB1}<{Ff_1}
	\ncline{->}{FA}{FB2}^{Ff_2}
	\ncline[linestyle=dashed]{->}{FB1}{FC}^{Fg_1}
	\ncline[linestyle=dashed]{->}{FB2}{FC}>{Fg_2}
	\ncarc[arcangle=25]{->}{FB1}{GT}^{h_1}
	\ncarc[arcangle=-25]{->}{FB2}{GT}>{h_2}
	\ncarc[linestyle=dashed,arcangle=30]{->}{FC}{GTp}^{h}
	\ncline[linestyle=dashed]{->}{GT}{GTp}_{Gk}
	\end{psmatrix}
	\]
\end{definition}

Now we are ready to link up our previous observations in the following result:
\begin{theorem}\label{mainconstruction}
	Let $\fA$ and $\fB$ be a $\lambda$-algebroidal categories all of whose morphisms are monos, and suppose that $\fB_{<\lambda}$ has the joint embedding property and the amalgamation property.   Let $\fC$ be any category and let 
	 $F:\fA\to\fC$, $G:\fB\to\fC$ be functors such that 
	\begin{enumerate}
		\item $F$ is $\lambda$-cocontinuous and $\mu_{<\lambda}$-cocontinuous for all $\mu<\lambda$,
		\item $F$ preserves $\lambda$-smallness with respect to $G$, and 
		\item $G$ is $\lambda_{<\lambda}$-cocontinuous, 
		\item $G$ preserves monos.
		\item for every $A\in\fA_{<\lambda}$ and for every $B\in \fB_{<\lambda}$ there are at most $\lambda$ morphisms in $\fC(FA\to GB)$.
	\end{enumerate}
	Then $(F \downarrow G)$ contains an $(F\downarrow G)$-universal, $(F\downarrow G)_{<\lambda}$-homogeneous object if and only if $\fA_{<\lambda}$ has the $(F\restr_{\fA_{<\lambda}},G\restr_{\fB_{<\lambda}})$-joint embedding property, and the $(F\restr_{\fA_{<\lambda}},G\restr_{\fB_{<\lambda}})$-amalgamation property. 
\end{theorem}
\begin{proof}
	By construction, all morphisms of $(F \downarrow G)$ are monomorphisms. 
	From Proposition~\ref{algebroidal}, it follows that $(F \downarrow G)$ is $\lambda$-algebroidal and all $\lambda$-small objects of $(F\downarrow G)$ are inherited. 
	
	``$\Rightarrow$'' From Theorem~\ref{DroGoe} it follows that $(F\downarrow G)_{<\lambda}$ has the joint embedding property and the amalgamation property. Let $B_1,B_2\in\fA_{<\lambda}$, $T\in\fB_{<\lambda}$, $h_1:FB_1\to GT$, $h_2:FB_2\to GT$. Then $(B_1,h_1,T), (B_2,h_2,T)\in (F\downarrow G)_{<\lambda}$. Since $(F\downarrow G)_{<\lambda}$ has the joint embedding property, it follows that there exists $(C,h,T')\in (F\downarrow G)_{<\lambda}$ and $(f_1,k_1):(B_1,h_1,T)\to (C,h,T')$ and $(f_2,k_2):(B_2,h_2,T)\to (C,h,T')$. That is, the following diagram commutes:
		\[
	\begin{psmatrix}
[name=FB1] FB_1 & [name=FC] FC & [name=FB2] FB_2\\
[name=GT1] GT & [name=GTp]GT' & [name=GT2] GT%
		\ncline{->}{FB1}{FC}_{Ff_1}%
		\ncline{->}{FB2}{FC}_{Ff_2}%
		\ncline{->}{FB1}{GT1}<{h_1}%
		\ncline{->}{FB2}{GT2}>{h_2}%
		\ncline{->}{GT1}{GTp}^{Gk_1}%
		\ncline{->}{GT2}{GTp}^{Gk_2}%
		\ncline{->}{FC}{GTp}<{h}%
	\end{psmatrix}
	\]
	Hence $\fA_{<\lambda}$ has the $(F\restr_{\fA_{<\lambda}},G\restr_{\fB_{<\lambda}})$-joint embedding property.
	
	Let $A,B_1,B_2\in\fA_{<\lambda}$, $f_1:A\to B_1$, $f_2:A\to B_2$, $T\in\fB_{<\lambda}$, $h_1:FB_1\to GT$, $h_2:FB_2\to GT$, such that $h_1\circ Ff_1 = h_2\circ Ff_2$. Then $(A, h_1\circ Ff_1,T), (B_1,h_1,T), (B_2,h_2,T)\in(F\downarrow G)_{<\lambda}$. Moreover, $(f_1,1_T):(A,h_1\circ Ff_1,T)\to (B_1,h_1,T)$ and $(f_2,1_T): (A,h_1\circ Ff_1,T)\to(B_2,h_2,T)$. Since $(F\downarrow G)_{<\lambda}$ has the amalgamation property, there exists $(C,h,T')\in(F\downarrow G)_{<\lambda}$ and $(g_1,k_1):(B_1,h_1,T)\to (C,h,T')$ and $(g_2,k_2):(B_2,h_2,T)\to (C,h,T')$ such that $(g_1,k_1)\circ (f_1,1_T) = (g_2,k_2)\circ (f_2,1_T)$. In particular, $k_1=k_2=:k$, and $g_1\circ f_1=g_2\circ f_2$. It follows that $\fA_{<\lambda}$ has the $(F\restr_{\fA_{<\lambda}},G\restr_{\fB_{<\lambda}})$-amalgamation property.
	
	``$\Leftarrow$'' Let $(B_1,h_1,T_1), (B_2,h_2,T_2)\in(F\downarrow G)_{<\lambda}$. Since $\fB_{<\lambda}$ has the joint embedding property, there exist $T\in\fB_{<\lambda}$, $g_1:T_1\to T$, and $g_2:T_2\to T$. Since $\fA_{<\lambda}$ has the $(F\restr_{\fA_{<\lambda}},G\restr_{\fB_{<\lambda}})$-joint embedding property, there exist $C\in\fA_{<\lambda}$, $T'\in\fB_{\lambda}$, $f_1:B_1\to C$, $f_2: B_2\to C$, $h:FC\to GT'$, and $k_1,k_2:T\to T'$ such that the following diagram commutes:
	\[
	\begin{psmatrix}
[name=FB1] FB_1 & [name=FC] FC & [name=FB2] FB_2\\
[name=GT1] GT & [name=GTp]GT' & [name=GT2] GT%
		\ncline{->}{FB1}{FC}_{Ff_1}%
		\ncline{->}{FB2}{FC}_{Ff_2}%
		\ncline{->}{FB1}{GT1}<{Gg_1\circ h_1}%
		\ncline{->}{FB2}{GT2}>{Gg_2\circ h_2}%
		\ncline{->}{GT1}{GTp}^{Gk_1}%
		\ncline{->}{GT2}{GTp}^{Gk_2}%
		\ncline{->}{FC}{GTp}<{h}%
	\end{psmatrix}
	\]
	In particular, $(f_1,k_1\circ g_1):(B_1,h_1,T_1)\to (C,h,T')$ and $(f_2,k_2\circ g_2):(B_2,h_2,T_2)\to (C,h,T')$, so $(F\downarrow G)_{<\lambda}$ has the joint embedding property.
	
	Let $(A,h,T),(B_1,h_1,T_1),(B_2,h_2,T_2)\in(F\downarrow G)_{<\lambda}$, and let $(f_1,g_1):(A,h,T)\to (B_1,h_1,T_1)$ and $(f_2,g_2):(A,h,T)\to (B_2,h_2,T_2)$. Since $\fB_{<\lambda}$ has the amalgamation property, there exists $T'\in\fB_{<\lambda}$ and $k_1:T_1\to T'$, $k_2:T_2\to T'$ such that $k_1\circ g_1=k_2\circ g_2$. Since $\fA_{<\lambda}$ has the $(F\restr_{\fA_{<\lambda}},G\restr_{\fB_{<\lambda}})$-amalgamation property, there exists $C\in\fA_{<\lambda}$, $l_1:B_1\to C$, $l_2:B_2\to C$, $T''\in\fB_{<\lambda}$, $k:T'\to T''$, $\hat{h}:FC\to GT''$ such that the following diagrams commute:  
		\[
	\begin{psmatrix}
	& & &  & & & [name=GTp] GT''\\
	& & &  & & [name=GT] GT'\\
	[name=B1] B_1 & [name=C] C & &  [name=FB1] FB_1 & [name=FC] FC\\
	[name=A] A & [name=B2]B_2 & & [name=FA] FA & [name=FB2] FB_2
	\ncline{->}{A}{B1}<{f_1}
	\ncline{->}{A}{B2}^{f_2}
	\ncline{->}{B1}{C}^{l_1}
	\ncline{->}{B2}{C}>{l_2}
	\ncline{->}{FA}{FB1}<{Ff_1}
	\ncline{->}{FA}{FB2}^{Ff_2}
	\ncline{->}{FB1}{FC}^{Fl_1}
	\ncline{->}{FB2}{FC}>{Fl_2}
	\ncarc[arcangle=25]{->}{FB1}{GT}^{Gk_1\circ h_1}
	\ncarc[arcangle=-25]{->}{FB2}{GT}>{Gk_2\circ h_2}
	\ncarc[arcangle=30]{->}{FC}{GTp}^{\hat{h}}
	\ncline{->}{GT}{GTp}_{Gk}
	\end{psmatrix}
	\]
	It remains to note that $(l_1,k\circ k_1):(B_1,h_1,T_1)\to(C,\hat{h},T'')$, $(l_2,k\circ k_2):(B_2,h_2,T_2)\to (C,\hat{h}, T'')$, and that $(l_1,k\circ k_1)\circ (f_1,g_1)= (l_2,k\circ k_2)\circ (f_2,g_2)$. It follows that $(F\downarrow G)_{<\lambda}$ has the amalgamation property. It follows from Theorem~\ref{DroGoe} that $(F\downarrow G)$ has an $(F\downarrow G)$-universal, $(F\downarrow G)_{<\lambda}$-homogeneous object. 
\end{proof}

The $(F,G)$-joint embedding property and the $(F,G)$-amalgamation property are somewhat technical conditions. Fortunately, in most of the interesting cases,  they follow from a stronger condition that is independent of $G$.
Recall that a \emph{weak pushout square} in a category $\fC$ is a square 
\begin{equation}\label{pushoutsquare}
\begin{psmatrix}
	[name=B1]B_1  & [name=C]C\\
	[name=A]A  &[name=B2]B_2
	\ncline{->}{A}{B1}<{f_1}
	\ncline{->}{A}{B2}_{f_2}
	\ncline{->}{B1}{C}^{g_1}
	\ncline{->}{B2}{C}>{g_2}
\end{psmatrix}
\end{equation}
such that $(C,(g_1,g_2))$ is a weakly limiting cocone of the diagram
\[
\begin{psmatrix}
	[name=B1]B_1 \\
	[name=A]A  &[name=B2]B_2
	\ncline{->}{A}{B1}<{f_1}
	\ncline{->}{A}{B2}_{f_2}
\end{psmatrix}
\]
In this case we call $(C,(g_1,g_2))$ a \emph{weak pushout} of the morphisms $f_1$ and $f_2$.
If $(C,(g_1,g_2))$ is even a limiting cocone, then we call \eqref{pushoutsquare} a \emph{pushout square}, and we call $(C,(g_1,g_2))$ a \emph{pushout} of $f_1$ and $f_2$. 
\begin{definition}
	Let $\fA,\fC$ be categories, $F:\fA\to\fC$. 
	We say that $\fA$ has the \emph{strict joint embedding property with respect to $F$} if for every $B_1,B_2\in\fA$ there exists a $C\in \fA$ and $f_1:B_1\to C$, $f_2:B_2\to C$ such that $(FC, (Ff_1,Ff_2))$ is a weak coproduct of $FB_1$ and $FB_2$ in $\fC$.
	
	We say that $\fA$ has the \emph{strict amalgamation property with respect to $F$} if for all $A,B_1,B_2\in\fA$, $f_1:A\to B_1$, $f_2:A\to B_2$ there exists $C\in\fA$ and $g_1:B_1\to C$, $g_2:B_2\to C$ such that $g_1\circ f_1= g_2\circ f_2$ and such that the following is a weak pushout-square in $\fC$:
	\[
	\begin{psmatrix}
		[name=FB1]FB_1 & [name=FC]FC \\
		[name=FA]FA & [name=FB2]FB_2
		\ncline{->}{FA}{FB1}<{Ff_1}%
		\ncline{->}{FA}{FB2}^{Ff_2}%
		\ncline{->}{FB1}{FC}_{Fg_1}%
		\ncline{->}{FB2}{FC}>{Fg_2}%
	\end{psmatrix}
	\]
\end{definition}
\begin{lemma}\label{strictAPJEP}
	Let $\fA,\fB,\fC$ be categories and $F:\fA\to\fC$, $G:\fB\to \fC$. Then 
	\begin{enumerate}
	\item\label{FGJEP} if $\fA$ has the strict joint embedding property with respect to $F$, then it has also the $(F,G)$-joint embedding property,
	\item\label{FGAP} if $\fA$ has the strict amalgamation property with respect to $F$, then it has also the $(F,G)$-amalgamation property.
	\end{enumerate}
\end{lemma}
\begin{proof}
	\textbf{About \eqref{FGJEP}:}
	Let $B_1,B_2\in\fA$, $T\in\fB$, $h_1:FB_1\to GT$, $h_2:FB_2\to GT$. Let further $C\in\fA$, $\iota_1:B_1\to C$, $\iota_2:B_2\to C$ such that $(FC,(F\iota_1,F\iota_2))$ is a weak coproduct of $FB_1$ and $FB_2$. Then there exists $h:FC\to GT$ such that the following diagram commutes:
\[	\begin{psmatrix}
		&   &  [name=GT]GT\\
	  [name=FB1]FB_1 & [name=FC]FC \\
		& [name=FB2]FB_2
		\ncline{->}{FB1}{FC}^{F\iota_1}
		\ncline{->}{FB2}{FC}>{F\iota_2}
		\ncarc[arcangle=20]{->}{FB1}{GT}^{h_1}
		\ncarc[arcangle=-20]{->}{FB2}{GT}>{h_2}
		\ncline{->}{FC}{GT}^h
	\end{psmatrix}
\]
But then also the following diagram commutes:
	\[
	\begin{psmatrix}
[name=FB1] FB_1 & [name=FC] FC & [name=FB2] FB_2\\
[name=GT1] GT & [name=GTp]GT & [name=GT2] GT%
		\ncline{->}{FB1}{FC}_{Ff_1}%
		\ncline{->}{FB2}{FC}_{Ff_2}%
		\ncline{->}{FB1}{GT1}<{h_1}%
		\ncline{->}{FB2}{GT2}>{h_2}%
		\ncline{->}{GT1}{GTp}^{G1_T}%
		\ncline{->}{GT2}{GTp}^{G1_T}%
		\ncline{->}{FC}{GTp}<{h}%
	\end{psmatrix}
	\]
	whence $\fA$ has the $(F,G)$-joint embedding property.
	
	\textbf{About \eqref{FGAP}:} Let $A,B_1,B_2\in\fA$, $f_1:A\to B_1$, $f_2:A\to B_2$, $T\in\fB$, $h_1:FB_1\to GT$, $h_2:FB_2\to GT$ such that $h\circ Ff_1=h_2\circ Ff_2$. Let $C\in\fA$, $g_1:B_1\to C$, $g_2:B_2\to C$, such that $g_1\circ f_1=g_2\circ f_2$, and such that $(FC, (Fg_1,Fg_2))$ is a pushout of $Ff_1$ and $Ff_2$ in $\fC$. Then there exists $h:FC\to GT$, such that the following diagram commutes:
\[	\begin{psmatrix}
		&   &  [name=GT]GT\\
	  [name=FB1]FB_1 & [name=FC]FC \\
		[name=FA]FA& [name=FB2]FB_2
		\ncline{->}{FB1}{FC}^{Fg_1}
		\ncline{->}{FB2}{FC}>{Fg_2}
		\ncarc[arcangle=20]{->}{FB1}{GT}^{h_1}
		\ncarc[arcangle=-20]{->}{FB2}{GT}>{h_2}
		\ncline{->}{FC}{GT}^h
		\ncline{->}{FA}{FB1}<{Ff_1}
		\ncline{->}{FA}{FB2}^{Ff_2}
	\end{psmatrix}
\]
However, then also the  following diagram commutes:
	\[
	\begin{psmatrix}
	& & & [name=GTp] GT\\
	& & [name=GT] GT\\
	[name=FB1] FB_1 & [name=FC] FC\\
	[name=FA] FA & [name=FB2] FB_2
	\ncline{->}{FA}{FB1}<{Ff_1}
	\ncline{->}{FA}{FB2}^{Ff_2}
	\ncline{->}{FB1}{FC}^{Fg_1}
	\ncline{->}{FB2}{FC}>{Fg_2}
	\ncarc[arcangle=25]{->}{FB1}{GT}^{h_1}
	\ncarc[arcangle=-25]{->}{FB2}{GT}>{h_2}
	\ncarc[arcangle=30]{->}{FC}{GTp}^{h}
	\ncline{->}{GT}{GTp}_{G1_T}
	\end{psmatrix}
	\]
	whence $\fA$ has the $(F,G)$-amalgamation property.
\end{proof}

A homogeneous object $(U,u,T)$ in a comma category $(F\downarrow G)$ is bound to have many symmetries. It is natural to ask what are the implications for the symmetries of the domain $U$ and the codomain $T$. Of particular interest to us is domain $U$. In the following we will answer the question when $U$ is a $\lambda$-saturated object in $\fA$.  
\begin{definition}
	Let $\fA,\fB,\fC$ be categories and let $F:\fA\to\fC$ and $G:\fB\to\fC$. We say that $F$ and $G$ have the \emph{mixed amalgamation property} if for all $A,B\in\fA$, $T_1\in\fB$, $g:A\to B$, $a:FA\to GT_1$, there exists $T_2\in \fB$, $h: T_1\to T_2$, and $b:FB\to GT_2$ such that the following diagram commutes:
	\[
		\begin{psmatrix}
		[name=FA]FA & [name=GT1]GT_1\\
		[name=FB]FB & [name=GT2]GT_2.
		\ncline{->}{FA}{GT1}^a%
		\ncline{->}{FA}{FB}<{Fg}%
		\ncline[linestyle=dashed]{->}{FB}{GT2}^{b}%
		\ncline[linestyle=dashed]{->}{GT1}{GT2}<{Gh}%
		\end{psmatrix}
	\]
\end{definition}

\begin{proposition}\label{AFsat}
	Let $\fA,\fB$ be $\lambda$-algebroidal categories, $\fC$ be a category. Let $F:\fA\to\fC$, $G:\fB\to\fC$ be  functors such that
	\begin{enumerate}
		\item $F$ is faithful,
		\item $F$ is $\lambda$-cocontinuous and $\mu_{<\lambda}$-cocontinuous for all $\mu<\lambda$,
		\item $F$ preserves $\lambda$-smallness with respect to $G$, and 
		\item $G$ is $\lambda_{<\lambda}$-cocontinuous, 
		\item $G$ preserves monos.
		\item for all $A\in\fA_{<\lambda}$, $B\in\fB_{<\lambda}$  there are at most $\lambda$ morphisms from $FA$ to $GB$ in $\fC$,
		\item $(F\downarrow G)_{<\lambda}$ has the joint embedding property and the amalgamation property. 
	\end{enumerate}
	Let $(U,u,T)$ be an $(F\downarrow G)$-universal, $(F\downarrow G)_{<\lambda}$-homogeneous object in $(F\downarrow G)$. Then $U$ is $\fA_{<\lambda}$-saturated if and only if $F\restr_{\fA_{<\lambda}}$ and $G\restr_{\fB_{<\lambda}}$ have the mixed amalgamation property. 
\end{proposition}
\begin{proof}
	By Proposition~\ref{algebroidal} we have that $(F\downarrow G)$ is $\lambda$-algebroidal, and by the proof of Proposition~\ref{semialgebroidal} we have that every $\lambda$-small object of $(F\downarrow G)$ is inherited.

	``$\Leftarrow$'': Let $A,B\in\fA_{<\lambda}$, $\iota_1\in\fA(A\to U)$, $\iota_2\in\fA(A\to B)$. Then $u\circ F\iota_1:FA\to GT$.  In particular, $(A,u\circ F\iota_1,T)\in(F\downarrow G)$. Since $(F\downarrow G)$ is $\lambda$-algebroidal, $(A,u\circ F\iota_1,T)$ is the colimit of some $\lambda$-chain $H$ of $\lambda$-small objects in $(F\downarrow G)$. Suppose $Hi=(A_i,h_i,T_i)$ and let $((A,u\circ F\iota_1,T),(g_i,g_i')_{i<\lambda})$ be a limiting cocone of $H$. Then, by the assumptions and by Corollary~\ref{colimcor}, we have that $(A,(g_i)_{i<\lambda})$ is a limiting cocone of $UH$ (where $U:(F\downarrow G) \to F$ is one of the forgetful functors of $(F\downarrow G)$). Consider the identity morphism $1_A: A\to A$. Since $A$ is  $\lambda$-small, there exists some $i<\lambda$ and some $f\in\fA(A\to A_i)$ such that $g_i\circ f = 1_A$. With this it becomes clear that $(1_A,g_i'):(A,h_i\circ Ff,T_i)\to(A,u\circ F\iota_1,T)$. Moreover, $(\iota_1,1_T):(A,u\circ F\iota_1,T)\to (U,u,T)$. Thus $(\iota_1,g_i'):(A,h_i\circ Ff,T_i)\to(U,u,T)$.
	
	Since $F\restr_{\fA_{<\lambda}}$ and $G\restr_{\fB_{<\lambda}}$ have the mixed amalgamation property, there exists  $\hat{T}_i\in\fB_{<\lambda}$, $h\in\fB(T_i\to\hat{T}_i)$, and $b\in\fC(FB\to G\hat{T}_i)$ such that the following diagram commutes:
	\[
	\begin{psmatrix}
	[name=FA] FA & [name=GTi] GT_i\\
	[name=FB] FB & [name=GhTi] G\hat{T}_i
	\ncline{->}{FA}{GTi}_{h_i\circ Ff}
	\ncline{->}{FB}{GhTi}^{b}
	\ncline{->}{FA}{FB}<{F\iota_2}
	\ncline{->}{GTi}{GhTi}>{Gh}
	\end{psmatrix}
	\]
	This means that $(\iota_2,h):(A,h_i\circ Ff,T_i)\to(B,b,\hat{T}_i)$. Since $(U,u,T)$ is $(F\downarrow G)$-universal and $(F\downarrow G)_{<\lambda}$-homogeneous, we obtain from Proposition~\ref{univhomsat} that it is $(F\downarrow G)_{<\lambda}$-saturated. It follows that there exists $(\hat{g},\hat{h}):(B,b,\hat{T}_i)\to (U,u,T)$, such that the following diagram commutes:
	\[
	\begin{psmatrix}
		[name=FU] FU & [name=GT] GT \\
		[name=FA] FA & [name=GT2] GT \\
		[name=FA2] FA & [name=GTi] GT_i\\
		[name=B] B & [name=GhTi]G\hat{T}_i
		\ncline{->}{FU}{GT}^u
		\ncline{->}{FA}{GT2}^{u\circ Fu}
		\ncline{->}{FA2}{GTi}^{h_i\circ Ff}
		\ncline{->}{B}{GhTi}^{b}
		\ncline{->}{FA}{FU}<{F\iota_1}
		\ncline{->}{FA2}{FA}<{F1_A}
		\ncline{->}{FA2}{B}<{F\iota_2}
		\ncline{->}{GT2}{GT}>{G1_T}
		\ncline{->}{GTi}{GT2}>{Gg_i'}
		\ncline{->}{GTi}{GhTi}>{Gh}
		\ncarc[arcangle=50]{->}{B}{FU}<{F\hat{g}}
		\ncarc[arcangle=-50]{->}{GhTi}{GT}>{G\hat{h}}
	\end{psmatrix}
	\]
	In particular, $F\iota_1= F\hat{g}\circ F\iota_2$. Since $F$ is faithful, it follows that $\iota_1=\hat{g}\circ\iota_2$. Hence $U$ is $\fA_{<\lambda}$-saturated.
	
	``$\Rightarrow$'': Suppose, $U$ is $\fA_{<\lambda}$-saturated. Let $A,B\in\fA_{<\lambda}$, $T_1\in\fB_{<\lambda}$, $a:FA\to GT_1$, and $g: A\to B$. Since $(U,u,T)$ is $(F\downarrow G)$-universal, there exists $(\iota,\iota'):(A,a,T_1)\to(U,u,T)$. Since $U$ is $\fA_{<\lambda}$-saturated, there exists $\hat{\iota}\in\fA(B\to U)$ such that $\iota=\hat{\iota}\circ g$. Hence the following diagram commutes:
	\[
	\begin{psmatrix}
		[name=FA]FA & [name=GT1]GT_1 \\
		[name=FB]FB & [name=GT]GT\\
		[name=FU]FU
		\ncline{->}{FA}{GT1}^{a}
		\ncline{->}{FB}{GT}^{u\circ F\hat\iota}
		\ncline{->}{FU}{GT}_{u}
		\ncline{->}{FA}{FB}<{Fg}
		\ncline{->}{FB}{FU}<{F\hat\iota}
		\ncline{->}{GT1}{GT}>{G\iota'}
	\end{psmatrix}
	\]
	Since $(F\downarrow G)$ is $\lambda$-algebroidal, $(B,u\circ F\hat\iota,T)$ is the colimit of a $\lambda$-chain $H$ of objects in $(F\downarrow G)_{<\lambda}$. Suppose $Hi=(B_i,h_i,\hat{T}_i)$  ($i<\lambda$), and  $((B,u\circ F\hat\iota,T),(g_i,g_i')_{i<\lambda})$ is a limiting cocone for $H$. Since $(A,a,T_1)$ is $\lambda$-small, there exists some $i<\lambda$ and some $(f,f'):(A,a,T_1)\to(B_i,h_i,\hat{T}_i)$ such that $(g_i,g_i')\circ (f,f')=(g,\iota')$. Since $B$ is $\lambda$-small, there exists some $j<\lambda$ and some $\hat{f}:B\to B_j$ such that $g_j\circ \hat{f}=1_B$. Since $g_j$ is a monomorphism (recall that all morphisms of $\fA$ are monos), from $g_j\circ\hat{f}\circ g_j=g_j$, it follows that $\hat{f}\circ g_j = 1_{B_j}$. Hence,  $g_j$ is an isomorphism and $\hat{f}=g_j^{-1}$. Without loss of generality we can assume that $i=j$. Now observe that the following diagram commutes:
	\[
	\begin{psmatrix}
	[name=FA]FA & [name=GT1]GT_1 \\
	[name=FBi]FB_i & [name=GhTi]G\hat{T}_i\\
	[name=FB]FB & [name=GhTi2]G\hat{T}_i\\
	[name=FB2]FB & [name=GT]GT
	\ncline{->}{FA}{GT1}_a
	\ncline{->}{FBi}{GhTi}_{h_i}
	\ncline{->}{FB}{GhTi2}_{h_i\circ F\hat{f}}
	\ncline{->}{FB2}{GT}^{u\circ F\hat\iota}
	\ncline{->}{FA}{FBi}<{Ff}
	\ncline{->}{FBi}{FB}<{Fg_i}
	\ncline{->}{FB}{FB2}<{F1_B}
	\ncline{->}{GT1}{GhTi}>{Gf'}
	\ncline{->}{GhTi}{GhTi2}>{G1_{\hat{T}_i}}
	\ncline{->}{GhTi2}{GT}>{Gg_i'}
	\end{psmatrix}
	\] 
	Since $g=g_i\circ f$, we have that the following diagram commutes:
	\[
	\begin{psmatrix}
	[name=FA]FA & [name=GT1]GT_1 \\
	[name=FB]FB & [name=GhTi2]G\hat{T}_i\\
	\ncline{->}{FA}{GT1}_a
	\ncline{->}{FB}{GhTi2}_{h_i\circ F\hat{f}}
	\ncline{->}{FA}{FB}<{Fg}
	\ncline{->}{GT1}{GhTi2}>{Gf'}
	\end{psmatrix}
	\] 	
	It follows that $F\restr_{\fA_{<\lambda}}$ and $G\restr_{\fB_{<\lambda}}$ have the mixed amalgamation property. 
\end{proof}

\section{Universal structures through universal homomorphisms}\label{s4}
In this section we will talk about our first application of Theorem~\ref{mainconstruction} --- the construction of universal homomorphisms and their use for the construction of universal structures.

Let us us first of all define  the objects of interest in this section --- the universal homogeneous homomorphisms. 
\begin{definition}\label{univhom}
	Let $\cC$ be an age, $\bU\in\overline{\cC}$, and let $\bT$ be a countable structure of the same type like $\bU$. A homomorphism $u:\bU\to\bT$ is called \emph{universal within $\overline{\cC}$} if for every $\bA\in\overline{\cC}$, and for every homomorphism $h:\bA\to\bT$ there exists an embedding $\iota:\bA\injto\bU$ such that $h=u\circ\iota$
	
	We call $u$ \emph{homogeneous} if for every finitely generated substructure $\bA\le\bU$ and for every embedding $\iota:\bA\to\bU$ with $u=u\circ\iota$ there exists an automorphism $\alpha$ of $\bU$ such that $u\circ\alpha=u$ and such that  $\alpha$ restricted to $A$ is equal to $\iota$.
\end{definition}

Before coming to the result about the existence of universal homogeneous homomorphisms, we have to introduce the notions of strict amalgamation classes in the sense of Dolinka (cf. \cite[Sec.1.1]{Dol12}). 

\begin{definition}
	Let $\cC$ be a \Fraisse-class. Then we say that $\cC$ has the \emph{strict amalgamation property} if for all $\bA,\bB_1,\bB_2\in\cC$, and for all embeddings $f_1:\bA\injto\bB_1$, $f_2:\bA\injto\bB_2$ there exists some $\bC\in\cC$ and homomorphisms $g_1:\bB_1\to\bC$, $g_2:\bB_2\to\bC$ such that the following is a pushout-square in $(\overline{\cC},\rightarrow)$:
	\[
	\begin{psmatrix}
		[name=B1]\bB_1 & [name=C]\bC\\
		[name=A]\bA & [name=B2] \bB_2.
		\ncline{H->}{A}{B1}<{f_1}
		\ncline{H->}{A}{B2}^{f_2}
		\ncline{->}{B1}{C}_{g_1}
		\ncline{->}{B2}{C}>{g_2}
	\end{psmatrix}
	\] 
	A \emph{strict \Fraisse-class} is a \Fraisse-class that enjoys the strict amalgamation property.
\end{definition}
Note that $g_1$ and $g_2$, if they exist, will always be embeddings. Thus the strict amalgamation property postulates canonical amalgams. Moreover, if $\cC$ has the strict amalgamation property then the category $(\cC,\injto)$ has the strict amalgamation property and the strict joint embedding property with respect to the identical embedding of $(\cC,\injto)$ into $(\overline{\cC},\to)$. 

If $\cU$ is a strict \Fraisse-class, then a \Fraisse-class $\cC$ that is a subclass of $\cU$,  will be called \emph{free in $\cC$} if it is closed with respect to canonical amalgams in $\cU$.

Note that every free amalgamation class of relational structures over the signature $R$, in our terminology, is a free \Fraisse-class in the class of all finite relational $R$-structures.  Moreover, every free amalgamation class is also a strict \Fraisse-class. However, there are strict \Fraisse-classes that are not free amalgamation classes. The class of finite partial orders is an example.   

\begin{theorem}\label{hom-constraints} 
	Let $\cU$ be a strict \Fraisse-class. Let $\bT\in\overline{\cU}$, and let $\cC\subseteq\cU$ be a \Fraisse-class that is free in $\cU$.  
	Then  there exists a universal homogeneous homomorphism $u:\bU\to \bT$ within $\overline\cC$.
	Moreover, if $\hat{u}:\hat\bU\to\bT$ is another such homomorphism, then there exists an isomorphism $h:\hat\bU\to\bU$ such that  $\hat{u}=u\circ h$. 
\end{theorem}
\begin{proof}
	Our goal is to invoke Theorem~\ref{mainconstruction}. We set $\fC:=(\overline\cU,\to)$, and $\fA=(\overline{\cC},\injto)$. Clearly, $\fA$ is $\aleph_0$-algebroidal. The $\aleph_0$-small objects in $\fA$ are just the finitely generated structures.
	
	Let $\fB$ be the category that has just one object $\bT$ and only the identity morphism $1_\bT$. Then $\fB$ is a trivial example of an $\aleph_0$-algebroidal, and clearly, $\fB_{<\aleph_0}=\fB$ has the amalgamation property and the joint embedding property.
	
	Let  $F: \fA\to\fC$, $G:\fB\to\fC$ be the identical embedding-functors.
	
	The conditions imposed onto $F$ and $G$ by Theorem~\ref{mainconstruction} are trivially fulfilled. 
	
	Since $\cU$ is a strict \Fraisse-class it follows that $\fA_{<\aleph_0}$ has the strict amalgamation property with respect to $F$. Moreover, since $\fA$ has an initial object (the structure that is generated by the empty set), and since $F$ maps this object to an initial object of $\fC$, it follows that $\fA_{<\aleph_0}$ also has the strict joint embedding property with respect to $F$. Hence, by Lemma~\ref{strictAPJEP}, it follows  that $\fA_{<\aleph_0}$ has the $(F\restr_{\fA_{\aleph_0}},G\restr_{\fB_{\aleph_0}})$ joint embedding property and the $(F\restr_{\fA_{\aleph_0}},G\restr_{\fB_{\aleph_0}})$ amalgamation property.  
	
	Now, by Theorem~\ref{mainconstruction}, there exists an $(F\downarrow G)$-universal, $(F\downarrow G)_{<\aleph_0}$-homogeneous object $(\bU,u,\bT)$ in $(F\downarrow G)$. It follows that  $u:\bU\to\bT$ is a universal homogeneous homomorphism to $\bT$ within $\overline{\cC}$. 
	
	Let $\hat{u}:\hat\bU\to\bT$ be another universal homogeneous homomorphism to $\bT$ within $\overline\cC$. Then $(\hat\bU,\hat{u},\bT)$ is an $(F\downarrow G)$-universal, $(F\downarrow G)_{<\aleph_0}$-homogeneous object in $(F\downarrow G)$. However, by Theorem~\ref{DroGoe},  there exists an isomorphism $(h,1_T):(\hat{\bU},\hat{u},\bT)\to(\bU,u,\bT)$ in $(F\downarrow G)$. This means that $h:\hat\bU\to\bU$ is an isomorphism such that $u\circ h=\hat{u}$.  
\end{proof}

\begin{lemma}\label{retract}
	With the notions from above, suppose that $\bT\in\overline{\cC}$.
	Let $u:\bU\to\bT$ be a unary  universal  homomorphism within  $\overline\cC$. Then $u$ is a retraction.
\end{lemma}
\begin{proof}
	Since $\bT\in\overline{\cC}$, and since $u$ is universal, there exists an embedding $\iota: \bT\injto\bU$ such that $u\circ\iota=1_\bT$.  Thus, $\iota$ is a right-inverse of $u$, and $u$ is a retraction. 
\end{proof}

Following we show on hands of a few examples, how universal homomorphisms give rise to universal structures:
\begin{example}
	Suppose, in Theorem~\ref{hom-constraints} we take $\bT\in\overline\cC$. The homomorphism equivalence class $\HE_{\overline{\cC}}(\bT)$ of $\bT$ in $\overline{\cC}$ is the class of all structures from $\overline{\cC}$ that are homomorphism-equivalent to $\bT$. Observe that  universal homogeneous homomorphism to $\bT$ within $\overline{\cC}$ gives rise to a universal object in the homomorphism-equivalence class of $\bT$ within $\overline{\cC}$. Indeed, if $u:\bU\to\bT$ is a  universal homogeneous homomorphism, and $\bA\in\overline\cC$ is homomorphism-equivalent with $\bT$, then there exists a homomorphism $h:\bA\to\bT$, and hence there exists an embedding $\iota:\bA\injto\bU$ such that $u\circ\iota=h$. In particular, $\bA$ embeds into $\bU$. On the other hand since $\bT\in\overline{\cC}$, we know from Lemma~\ref{retract} that $\bT$ is a retract of $\bU$. In particular, $\bU$ is homomorphism-equivalent with $\bT$. Hence it is a universal element in $\HE_{\overline{\cC}}(\bT)$.
\end{example}
Recall that a countable structure $\bU$ is called \emph{$\aleph_0$-categorical} if every other countable structure with the same first order theory like $\bU$ is isomorphic to $\bU$. The well-known Ryll-Nardzewski Theorem characterizes the $\aleph_0$-categorical structures as those structures that have an oligomorphic automorphism group (cf.\cite[Thm.6.3.1]{Hod97}).  Here, a permutation group $G$ on a set $\Omega$ is called \emph{oligomorphic} if for every $n\in\bN\setminus\{0\}$, the coordinate wise action of $G$ on $\Omega^n$ has just finitely many orbits (cf. \cite{Cam90}).

\Fraisse-limits are structures of exceptional symmetry. Often they have an oligomorphic automorphism group. This is for instance the case for the \Fraisse-limits of \Fraisse-classes of relational structures over finite signatures (like, e.g. $(\bQ,<)$, the Rado-graph, the universal homogeneous poset, etc.). 
Yet, the next example will show that in some cases the domain of a universal homogeneous homomorphism will not have an oligomorphic automorphism group --- even if we work over a finite relational signature:
\begin{example}
	A well-founded poset is a poset that does not contain infinite properly descending chains. 
	Let  $(P,<)$ be a well-founded poset (for technical reasons we work with strict posets). We may define a height function that assigns to each element of $P$ an ordinal number --- its height (cf. \cite[2.7.1]{Fra00}. The minimal elements of $(P,<)$ have height $0$. When removing all minimal elements from $P$ we obtain a new well-founded poset $(P',<)$. The minimal elements of this poset will be the elements of height $1$ in the original poset $(P,<)$. Proceeding by transfinite induction, to each element of $P$ a height is assigned. Note that the set of all heights of elements from $P$ is itself an ordinal number. This number is called the height of $(P,<)$ and is denoted by $\Ht(P,<)$. The sketched construction of a hight function is at the same time a proof that from every well-founded poset $(P,<)$ there exists a homomorphism to $(\Ht(P,<),<)$. On the other hand, for some poset $(Q,<)$ there is a homomorphism to some $(\alpha,\in)$, where $\alpha$ is an ordinal number, then $(Q,<)$ is well-founded of height $\le\alpha$.  With these remarks in mind we can construct universal well-founded posets of bounded height (we do so only for countable heights, but Theorem~\ref{mainconstruction} allows in principle a construction for larger heights):
	
	Let $\alpha$ be a countable ordinal number. Let $\bT=(\alpha,<)$. Take as $\cU$ the class of finite partial orders. And define $\cC:=\cU$. Clearly, $\cU$ is a strict \Fraisse-class. So by Theorem~\ref{hom-constraints}, there exists a universal homogeneous homomorphism $u:\bU\to\bT$ within $\overline{\cC}$. The poset $\bU$ is then a universal element in the class of all countable well-founded posets of height $\le\alpha$. Note that if $\alpha$ is infinite, then $\bU$ can not be $\aleph_0$-categorical, since automorphisms of $\bU$ have to preserve the height of elements.
\end{example}
 In the following we derive sufficient conditions for the domain of a universal homogeneous homomorphism to have an oligomorphic automorphism group.

Recall that a structure $\bA$ is \emph{loclally finite} if every finitely generated substructure of $\bA$ is finite. Moreover, $\bA$ is called \emph{uniformly locally} finite if there exists a function $\chi:\aleph_0\to\aleph_0$ such that every $n$-generated substructure of $\bA$ has size at most $\chi(n)$. Recall also that every countable homogeneous uniformly locally finite structure is $\aleph_0$-categorical. 
\begin{proposition}\label{oligomorphic}
	Let $\cU$ be an age,  let $\cC\subseteq\cU$ be a \Fraisse-class whose \Fraisse-limit  is  locally finite and has an oligomorphic automorphism group. Let $\bT\in\overline{\cU}$. Finally, let $u:\bU\to \bT$ be a universal homogeneous homomorphism to $\bT$ within $\overline{\cC}$. If $\Aut(\bT)$ is oligomorphic, then $\Aut(\bU)$ is oligomorphic, too.
\end{proposition}
For the proof of this proposition we need some preparation. 
Let $\hat\bU\in\overline{\cC}$ and let $\hat{u}:\hat\bU\to\bT$. We say that $\hat{u}$ is \emph{weakly universal within $\overline{\cC}$} if for all $\bA\in\overline{\cC}$ and for all $a:\bA\to\bT$ there exists an embedding $\iota:\bA\injto\hat\bU$ and some $g\in\Aut(\bT)$, such that $\hat{u}\circ \iota= g\circ a$. 

Let now $\bA\le\hat\bU$ and let $\iota:\bA\injto\hat\bU$ be an embedding. Then we say that $\iota$ \emph{weakly preserves $\hat{u}$} if for some $g\in\Aut(\bT)$ we have that $\hat{u}\circ \iota=g\circ \hat{u}$. We call $\hat{u}$ \emph{weakly homogeneous} if for all finitely generated substructures $\bA$ of $\hat\bU$ and for all weakly $\hat{u}$ preserving embeddings $\iota:\bA\injto\hat\bU$ there exists an automorphism $h$ of $\hat\bU$ that weakly preserves $\hat{u}$ and that extends $\iota$. 
\begin{lemma}\label{weakhom}
	With the notions from above, let $u:\bU\to\bT$ be a  universal homogeneous homomorphism within $\overline{\cC}$. Then $u$ is also weakly universal within $\overline{\cC}$ and weakly homogeneous. 
\end{lemma}
\begin{proof}
	Clearly, from universality follows weak universality. 
	
	In the following we show the weak homogeneity of $u$. Consider the categories $\fA,\fB,\fC$ and the functors $F,G$ from the proof of Theorem~\ref{hom-constraints}. Let $H$ be a countable subgroup of $\Aut(\bT)$. Define $\fB_H$ to be the category that has only one object $\bT$ and whose morphisms are the elements of $H$. Let $G_H:\fB_H\to\fC$ be the identical embedding. Then  $F$ and $G_H$ fulfill the conditions (1)--(5) of Theorem~\ref{mainconstruction}. The $(F\restr_{\fA_{<\aleph_0}},G_H\restr_{(\fB_H)_{<\aleph_0}})$-joint embedding property and the $(F\restr_{\fA_{<\aleph_0}},G_H\restr_{(\fB_H)_{<\aleph_0}})$-amalgamation property of $\fA$ follow directly from the $(F\restr_{\fA_{<\aleph_0}},G\restr_{\fB_{<\aleph_0}})$-joint embedding property and the $(F\restr_{\fA_{<\aleph_0}},G\restr_{\fB_{<\aleph_0}})$-amalgamation property of $\fA$, respectively. So all of the conditions of Theorem~\ref{mainconstruction} are fulfilled and we have that there exists an $(F\downarrow G_H)$-universal $(F\downarrow G_H)_{<\aleph_0}$-homogeneous object $(\bU_H,u_H,\bT)$ in $(F\downarrow G_H)$. Moreover, up to isomorphism in $(F\downarrow G_H)$ there is just one such object. Note that $(F\downarrow G)$ is a subcategory of $(F\downarrow G_H)$ and object-wise the two categories are indistinguishable. 
	
	We will show that $(\bU_H,u_H,\bT)$ is an $(F\downarrow G)$-universal and $(F\downarrow G)_{<\aleph_0}$-homogeneous object. 
	
	By assumption, we have that $(\bU,u,\bT)$ is an $(F\downarrow G)$-universal, $(F\downarrow G)_{<\aleph_0}$-homogeneous object. Since $(\bU_H,u_H,\bT)$ is $(F\downarrow G_H)$-universal, there exists an $(F\downarrow G_H)$-morphism $(f,g):(\bU,u,\bT)\to(\bU_H,u_H,\bT)$. Let $(\bA,a,\bT)\in(F\downarrow G)$. Then there exists an $(F\downarrow G)$-morphism $(\iota,1_T):(\bA,g^{-1}\circ a,\bT)\to(\bU,u,\bT)$. However, then $(f\circ\iota,1_\bT):(\bA,a,\bT)\to(\bU_H,u_H,\bT)$ in $(F\downarrow G)$. This shows that $(\bU_H,u_H,\bT)$ is $(F\downarrow G)$-universal. 
	
	Let  now $(\bA,a,\bT)\in(F\downarrow G)_{<\aleph_0}$, $(f_1,1_\bT),(f_2,1_\bT):(\bA,a,\bT)\to(\bU_H,u_H,\bT)$. Since $(\bU_H,u_H,\bT)$ is $(F\downarrow G_H)_{<\aleph_0}$-homogeneous, there exists an $(F\downarrow G_H)$-auto\-morphism $(h,g)$ of $(\bU_H,u_H,\bT)$ such that $(h,g)\circ(f_1,1_\bT)=(f_2,1_\bT)$. In particular $g\circ 1_\bT=1_\bT$. It follows that $g=1_\bT$. This shows that $(\bU_H,u_H,\bT)$ is $(F\downarrow G)_{<\aleph_0}$-homogeneous. From Theorem~\ref{mainconstruction} it follows that $(\bU,u,\bT)$ and $(\bU_H,u_H,\bT)$ are isomorphic in $(F\downarrow G)$, and hence also in $(F\downarrow G_H)$. In particular, it follows that $(\bU,u,\bT)$ is $(F\downarrow G_H)$-universal and $(F\downarrow G_H)_{<\aleph_0}$-homogeneous. 
	
	Now we are ready to show that $u$ is weakly homogeneous: Let $\bA\le\bU$ be finitely generated, let $f:\bA\injto\bU$ be an embedding that weakly preserves $u$. This means that there is some $g\in\Aut(\bT)$ such that  $u\circ f=g\circ u$. Let $H$ be the subgroup of $\Aut(\bT)$ that is generated by $g$. Then $H$ is countable. Let $\iota:\bA\injto\bU$ be the identical embedding, and let $a$ be the restriction of $u$ to $A$. Then $(\bA,a,\bT)\in(F\downarrow G_H)_{<\aleph_0}$, and $(\iota,1_\bT),(f,g):(\bA,a,\bT)\to(\bU,u,\bT)$ are morphisms in $(F\downarrow G_H)$. Since $(\bU,u,\bT)$ is $(F\downarrow G_H)_{<\aleph_0}$-homogeneous, it follows that there is an automorphism $(h,\hat{g})$ of $(\bU,u,\bT)$ in $(F\downarrow G_H)$, such that $(h,\hat{g})\circ(\iota,1_\bT)=(f,g)$. In particular we obtain $\hat{g}=g$ and $h\circ\iota=f$. That means that $h$ extends $f$ and that $u\circ h=g\circ u$. It follows that $u$ is weakly homogeneous.   
\end{proof}

\begin{proof}[Proof of Proposition~\ref{oligomorphic}]
	Since the \Fraisse-limit $\bF$ of $\cC$ is locally finite and since its automorphism group is oligomorphic, we have that $\bF$ is uniformly locally finite. Let $\chi:\aleph_0\to\aleph_0$ such that the size of every $n$-generated structure from $\cC$ is less than or equal $\chi(n)$.
	
	Let $\ba=(a_1,\dots,a_n)$, and $\bb=(b_1,\dots,b_n)$ be tuples of elements from $U$. Let $\bA$ and $\bB$ be the substructures of $\bU$ generated by the entries of $\ba$ and $\bb$, respectively. Suppose, the mapping $a_i\mapsto b_i$ induces an isomorphism $f:\bA\to\bB$, and suppose further that there exists a $g\in \Aut(\bT)$ such that we have that $g\circ u= u\circ f$.    
	
	Let $a:\bA\to\bT$ be the restriction of $u$ to $\bA$. Define $\iota:\bA\to\bU$ by $a_i\mapsto b_i$. Then, by the assumptions on $\ba$ and $\bb$ we have that $\iota$ is well-defined and weakly preserves $u$. By Lemma~\ref{weakhom}, it follows that $u$ is weakly homogeneous. Hence there exists an automorphism $\hat{f}$ of $\bU$ that weakly preserves $u$ and that extends $\iota$. In particular, $\hat{f}(\ba)=\bb$. Since there are only finitely many isomorphism types of $n$-tuples in $\bU$, and since there are only finitely many orbits of $\chi(n)$-tuples in $\bT$ under $\Aut(\bT)$, it follows that $\Aut(\bU)$ has only finitely many $n$-orbits. Hence, $\Aut(\bU)$ is oligomorphic.
\end{proof}

\begin{example}
	Consider the class $\cG_k$ of graphs whose chromatic number is bounded from above by $k$. Then this class coincides with the class of all those graphs that have a homomorphism to the complete graph $K_k$. The class of all graphs is a free amalgamation class, and hence it is also a strict \Fraisse-class. Taking $\cU=\cC$ to be the class of all finite graphs, from Theorem~\ref{hom-constraints} it follows that there exists a universal homogeneous  homomorphism $u:\bU\to\bT$. Hence $\bU$ is universal in $\cG_k$. 
	
	The \Fraisse-limit of $\cC$ is the Rado-graph, and we know that this graph is $\aleph_0$-categorical. The complete graph $K_n$ is finite. Hence, by Proposition~\ref{oligomorphic}, we have that $\bU$ has an oligomorphic automorphism group. Clearly, $\bU$ is not finite, so it has to be $\aleph_0$-categorical.

	These arguments function in the same way if we replace $K_n$ by any countable graph $H$ that has an oligomorphic automorphism group. Thus, from the Ryll-Nardzewski-Theorem, we obtain that there is an $\aleph_0$-categorical countable universal $H$-colorable graph.
\end{example}

\begin{example}
	Let $\bT$ be a finite or $\aleph_0$-categorical relational structure, and let $\HE(\bT)$ be the class of all countable structures of the same type like $\bT$ that are homomorphism-equivalent with $\bT$. By Theorem~\ref{hom-constraints}, there exists a universal homogeneous  homomorphism $u:\bU\to\bT$ within the class of all countable structures of the same type like $\bT$. We already observed, that $\bU$ is a universal element in $\HE(\bT)$. In this particular case, from Proposition~\ref{oligomorphic} and from the Ryll-Nardzewski-Theorem we obtain, that $\bU$ is $\aleph_0$-categorical.
\end{example}

\begin{example}
	A directed acyclic graph (or DAG, for short) is a simple digraph that contains no directed cycles (including loops and undirected edges). The transitive closure  of a DAG is a poset. In particular, the rationals with their strict order can be considered as a DAG and a simple digraph is a DAG if and only if it has a homomorphism into $(\bQ,<)$.  

	It is not hard to see that the class of all finite simple digraphs is a strict \Fraisse-class. We take it as the class $\cU$, and we define $\cC:=\cU$. Finally let us define $\bT:=(\bQ,<)$. By Theorem~\ref{hom-constraints}, there is an universal homogeneous  homomorphism $u:\bU\to\bT$ within the class of countable simple directed graphs. Hence $\bU$ is a universal object in the class of all countable DAGs. The \Fraisse-limit of the class of finite simple directed graphs is $\aleph_0$-categorical, and so is $(\bQ,<)$. Hence, by Proposition~\ref{oligomorphic} and by the Ryll-Nardzewski-Theorem, $\bU$ is $\aleph_0$-categorical, too. 
\end{example}

\section{Retracts of homogeneous structures}\label{s5}

Recall that a homomorphism $r:\bA\to\bB$ is called a \emph{retraction} if there exists a homomorphism $\iota:\bB\to\bA$ such that $r\circ\iota$ is the identity homomorphism of $\bB$. Clearly, $\iota$ must be an embedding, and $r$ must be surjective. If $\bB$ is actually a substructure of $\bA$, then we call $\bB$ a \emph{retract} of $\bA$. Equivalently we can say that $\bB$ is a retract of $\bA$ if and only if there is an idempotent endomorphism $h$ of $\bA$ whose image is $B$. 

In this section we will be interested in such retractions that are at the same time universal homogeneous homomorphisms:
\begin{definition}
	Let $\bA$ be a structure and let $\bB\in\overline{\Age(\bA)}$. A retraction $r:\bA\to\bB$ is called a \emph{universal homogeneous retraction} if it is a  universal homogeneous homomorphism  to $\bB$ within $\overline{\Age(\bA)}$ (in the sense of Definition~\ref{univhom}).
\end{definition}
We are going to use \Fraisse-limits in comma-categories in order to derive a characterization of retracts of homogeneous structures that are induced by universal homogeneous retractions and their subretracts.  Moreover, we will characterize all homogeneous structures that have the property the every retract is induced by a universal homogeneous retraction. Our results are related to \cite{Dol12} and \cite{Kub13}.

\begin{theorem}\label{uhretract}
	Let $\cC$ be a \Fraisse-class with \Fraisse-limit $\bU$, and let $\bT\in\overline\cC$. Then there exists a universal homogeneous retraction $r:\bU\epito\bT$ if and only if
	\begin{enumerate}
		\item\label{cond1} for all $\bA,\bB_1,\bB_2\in\cC$, $f_1:\bA\injto\bB_1$, $f_2:\bA\injto\bB_2$, $h_1:\bB_1\to\bT$, $h_2:\bB_2\to\bT$, if $h_1\circ f_1=h_2\circ f_2$, then there exists $\bC\in\cC$, $g_1:\bB_1\injto\bC$, $g_2:\bB_2\injto\bC$, $h:\bC\to\bT$, such that the following diagram commutes:
		\[
		\begin{psmatrix}
		& & [name=T] \bT\\
		[name=B1]\bB_1 & [name=C]\bC\\
		[name=A]\bA & [name=B2]\bB_2
		\ncline{H->}{A}{B1}<{f_1}
		\ncline{H->}{A}{B2}^{f_2}
		\ncline[linestyle=dashed]{H->}{B1}{C}_{g_1}
		\ncline[linestyle=dashed]{H->}{B2}{C}<{g_2}
		\ncarc[arcangle=25]{->}{B1}{T}^{h_1}
		\ncarc[arcangle=-25]{->}{B2}{T}>{h_2}
		\ncline[linestyle=dashed]{->}{C}{T}^{h}
		\end{psmatrix}
		\]
		\item\label{cond2} for all $\bA,\bB\in\cC$, $\iota:\bA\injto\bB$, $h:\bA\to\bT$ there exists $\hat{h}:\bB\to\bT$ such that the following diagram commutes:
		\[
		\begin{psmatrix}
		[name=A]\bA & [name=T]\bT\\
		[name=B]\bB
		\ncline{H->}{A}{B}<{\iota}
		\ncline{->}{A}{T}^{h}
		\ncline[linestyle=dashed]{->}{B}{T}_{\hat{h}}
		\end{psmatrix}
		\]
	\end{enumerate}
\end{theorem}
\begin{proof}
	Let $\fA=(\overline\cC,\injto)$, $\fC=(\overline\cC,\to)$, and let $\fB$ be the subcategory of $\fC$ that consists only of the object $\bT$ and the identity morphism $1_\bU$. Let $F:\fA\to\fC$ and $G:\fB\to\fC$ be identical embeddings. It is easy to check that conditions (1)--(5) of Theorem~\ref{mainconstruction} are fulfilled. Moreover, $F$ is obviously faithful.
	In particular, by Proposition~\ref{algebroidal}, $(F\downarrow G)$ is $\aleph_0$-algebroidal. Moreover, all morphisms of $(F\downarrow G)$ are monos.
	
	``$\Rightarrow$'' If $r:\bU\epito\bT$ is a universal homogeneous retraction, then $(\bU,r,\bT)$ is a $(F\downarrow G)$-universal $(F\downarrow G)_{<\aleph_0}$-homogeneous object in $(F\downarrow G)$. In particular, from Theorem~\ref{DroGoe} it follows that $(F\downarrow G)_{<\aleph_0}$ has the joint embedding property and the amalgamation property. Moreover,  by Theorem~\ref{mainconstruction}, we have that $\fA$ has the $(F\restr_{\fA_{<\aleph_0}},G\restr_{\fB_{<\aleph_0}})$-amalgamation property. From this condition (1) follows. Since $\bU$ is $\fA$-universal and $\fA_{<\aleph_0}$-homogeneous, it follows from Proposition~\ref{AFsat}, that $F$ and $G$ have the mixed amalgamation property. However, this implies condition (2).
	
	``$\Leftarrow$'' Condition (1) implies that $\fA$ has the $(F\restr_{\fA_{<\aleph_0}},G\restr_{\fB_{<\aleph_0}})$-amalgamation property. Since $\cC$ is the age of a structure, it follows that $\fA$ has an initial object (the substructure generated by $\emptyset$). Moreover, $F$ maps the initial structure of $\fA$ to the initial object  of $\fC$. Hence,  from the $(F\restr_{\fA_{<\aleph_0}},G\restr_{\fB_{<\aleph_0}})$-amalgamation property follows the $(F\restr_{\fA_{<\aleph_0}},G\restr_{\fB_{<\aleph_0}})$-joint embedding property. By Theorem~\ref{mainconstruction} we have that $(F\downarrow G)$  has an    $(F\downarrow G)$-universal, $(F\downarrow G)_{<\aleph_0}$-homogeneous object $(\hat\bU,r,\bT)$. 
	
	From condition (2) it follows that $F$ and $G$ have the mixed amalgamation property.  Moreover, since $\fC$ has an initial object, it follows that $\Age(\hat\bU)=\cC$. From Proposition~\ref{AFsat}, it follows that $\hat\bU$ is homogeneous. From \Fraisse's theorem we get that $\hat\bU$ is isomorphic with $\bU$, so without loss of generality, we can assume $\hat\bU=\bU$. But then, by Lemma~\ref{retract},  $r$ is the wanted universal homogeneous retraction.
\end{proof}

In the following we will need the notion of the amalgamated extension property that was defined by Kubi\'s \cite{Kub13}:
\begin{definition}
	Let $\cC$ be a class of countable, finitely generated structures. We say that $\cC$  has the \emph{amalgamated extension property} if for all $\bA,\bB_1,\bB_2,\bT\in\cC$, $f_1:\bA\injto\bB_2$, $f_2:\bA\injto\bB_2$, $h_1:\bB_1\to\bT$, $h_2:\bB_2\to\bT$, if $h_1\circ f_1=h_2\circ f_2$, then there exists $\bC,\bT'\in\cC$, $g_1:\bB_1\injto\bC$, $g_2:\bB_2\injto\bC$, $h:\bC\to\bT'$, $k:\bT\injto\bT'$ such that the following diagram commutes:
	
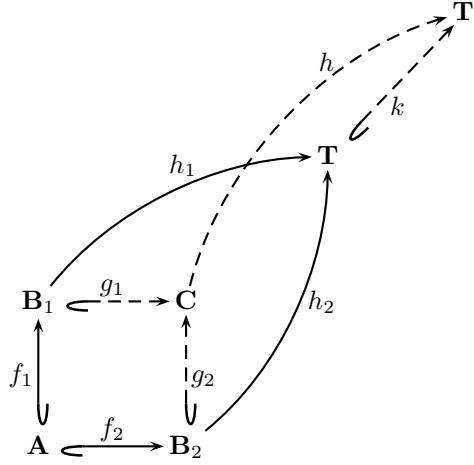
\[
	\begin{psmatrix}
	& & & [name=GTp] \bT'\\
	& & [name=GT] \bT\\
	[name=FB1] \bB_1 & [name=FC] \bC\\
	[name=FA] \bA & [name=FB2] \bB_2
	\ncline{H->}{FA}{FB1}<{f_1}
	\ncline{H->}{FA}{FB2}^{f_2}
	\ncline[linestyle=dashed]{H->}{FB1}{FC}^{g_1}
	\ncline[linestyle=dashed]{H->}{FB2}{FC}>{g_2}
	\ncarc[arcangle=25]{->}{FB1}{GT}^{h_1}
	\ncarc[arcangle=-25]{->}{FB2}{GT}>{h_2}
	\ncarc[linestyle=dashed,arcangle=30]{->}{FC}{GTp}^{h}
	\ncline[linestyle=dashed]{H->}{GT}{GTp}_{k}
	\end{psmatrix}
	\]
\end{definition}
Note that every strict \Fraisse-class has the amalgamated extension property. However, the opposite is not true. The following \Fraisse-classes have the amalgamated extension property, but fail to be strict \Fraisse-classes (see \cite{Kub13} for a detailed explanation):
\begin{itemize}
	\item the class of finite chains,
	\item the class of finite metric spaces with rational distances,
	\item the class of finite metric spaces with rational distances $\le 1$. 
\end{itemize}

\begin{lemma}[cf. {\cite[Lem.2.5]{Kub13}}]\label{amalgextred}
	With the notions from above, if $\cC$ has the amalgamated extension property, then condition \eqref{cond1} of Theorem~\ref{uhretract} follows from condition \eqref{cond2}.
\end{lemma}
\begin{proof}
	Let $\bA,\bB_1,\bB_2\in\cC$, $f_1:\bA\injto\bB_1$, $f_2:\bA\injto\bB_2$, $h_1:\bB_1\to\bT$, $h_2:\bB_2\to\bT$, such that $h_1\circ f_1=h_2\circ f_2$. Since $\bB_1$ and $\bB_2$ are finitely generated, there exists a finitely generated substructure $\bT'$  of $\bT$ such that $h_1$ and $h_2$ factor over the identical embedding $\iota:\bT'\injto\bT$. That is, there exist $h'_1:\bB_1\to\bT'$ and $h'_2:\bB_2\to\bT'$ such that $h_1=\iota\circ h'_1$ and $h_2=\iota\circ h'_2$. But then (since $\iota$ is a monomorphism) $h'_1\circ f_1=h'_2\circ f_2$. Since $\cC$ has the amalgamated extension property, there exists $\bC,\bT''\in\cC$, $g_1:\bB_1\injto\bC$, $g_2:\bB_2\injto\bC$, $h':\bC\to\bT''$, $k:\bT'\injto\bT''$ such that the following diagram commutes:
	
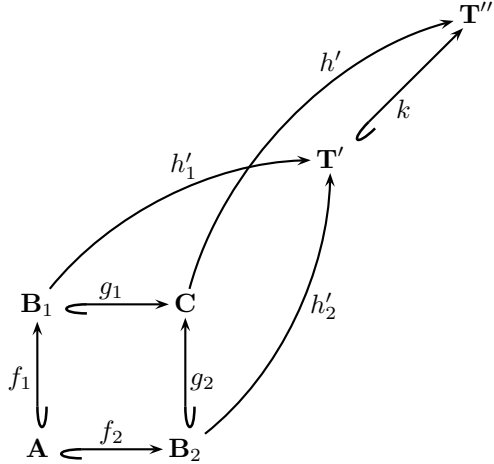
\[
	\begin{psmatrix}
	& & & [name=GTp] \bT''\\
	& & [name=GT] \bT'\\
	[name=FB1] \bB_1 & [name=FC] \bC\\
	[name=FA] \bA & [name=FB2] \bB_2
	\ncline{H->}{FA}{FB1}<{f_1}
	\ncline{H->}{FA}{FB2}^{f_2}
	\ncline{H->}{FB1}{FC}^{g_1}
	\ncline{H->}{FB2}{FC}>{g_2}
	\ncarc[arcangle=25]{->}{FB1}{GT}^{h'_1}
	\ncarc[arcangle=-25]{->}{FB2}{GT}>{h'_2}
	\ncarc[arcangle=30]{->}{FC}{GTp}^{h'}
	\ncline{H->}{GT}{GTp}_{k}
	\end{psmatrix}
	\]
	By condition (2), there exists a homomorphism $l:\bT''\to\bT$ such that $\iota=l\circ k$. That is, the following diagram commutes:
	
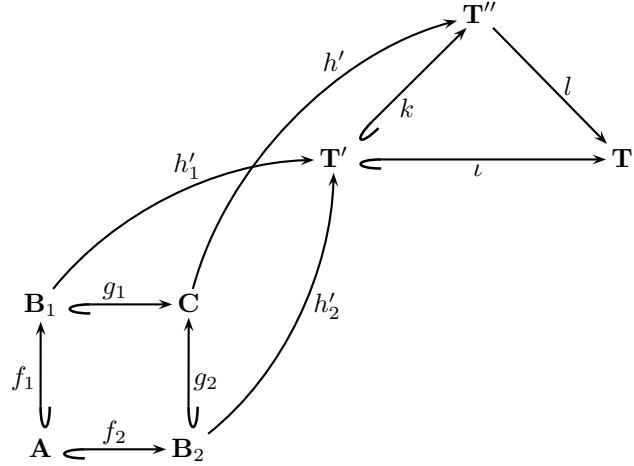
\[
	\begin{psmatrix}
	& & & [name=GTp] \bT''\\
	& & [name=GT] \bT' & & [name=T]\bT\\
	[name=FB1] \bB_1 & [name=FC] \bC\\
	[name=FA] \bA & [name=FB2] \bB_2
	\ncline{H->}{FA}{FB1}<{f_1}
	\ncline{H->}{FA}{FB2}^{f_2}
	\ncline{H->}{FB1}{FC}^{g_1}
	\ncline{H->}{FB2}{FC}>{g_2}
	\ncarc[arcangle=25]{->}{FB1}{GT}^{h'_1}
	\ncarc[arcangle=-25]{->}{FB2}{GT}>{h'_2}
	\ncarc[arcangle=30]{->}{FC}{GTp}^{h'}
	\ncline{H->}{GT}{GTp}_{k}
	\ncline{H->}{GT}{T}_{\iota}
	\ncline{->}{GTp}{T}>{l}
	\end{psmatrix}
	\]
	In particular, the following diagram commutes:
	
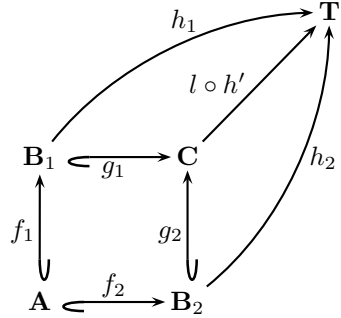
\[
	\begin{psmatrix}
	& & [name=T] \bT\\
	[name=B1]\bB_1 & [name=C]\bC\\
	[name=A]\bA & [name=B2]\bB_2
	\ncline{H->}{A}{B1}<{f_1}
	\ncline{H->}{A}{B2}^{f_2}
	\ncline{H->}{B1}{C}_{g_1}
	\ncline{H->}{B2}{C}<{g_2}
	\ncarc[arcangle=25]{->}{B1}{T}^{h_1}
	\ncarc[arcangle=-25]{->}{B2}{T}>{h_2}
	\ncline{->}{C}{T}<{l\circ h'}
	\end{psmatrix}
	\]
	Hence, condition (1) is fulfilled.
\end{proof}

The set of retracts of a structure $\bU$ can be endowed with a binary relation $\sqsubseteq$. We write $\bA\sqsubseteq\bB$ whenever $\bA$ is a retract of $\bB$. Clearly, $\sqsubseteq$ is a partial order relation. The following proposition shows that the universal homogeneous retracts of a countable homogeneous structure form a down-set in this poset:
\begin{proposition}\label{subretract}
	Let $\cC$ be a \Fraisse-class with \Fraisse-limit $\bU$  and let $\bV, \bW\in\overline{\cC}$
	Let $r:\bU\epito \bV$ be a universal homogeneous retraction. Let $s:\bV\epito\bW$ be any retraction. Then there is a universal homogeneous retraction $\hat{s}:\bU\epito\bW$.
\end{proposition}
\begin{proof}
	Let $t$ be a section for $s$ (i.e., $s\circ t=1_\bW$).
	
	We are going to show that conditions \eqref{cond1} and \eqref{cond2} of Theorem~\ref{uhretract} are fulfilled for $\bW$.
	
	Let $\bA,\bB_1,\bB_2\in\cC$, $f_1:\bA\injto\bB_1$, $f_2:\bA\injto\bB_2$, $h_1:\bB_1\to\bW$, and $h_2:\bB_2\to\bW$, such that $h_1\circ f_1=h_2\circ f_2$. Let $\hat h_1:=t\circ h_1$, and $\hat h_2:= t\circ h_2$. Then, obviously, $\hat h_1\circ f_1=\hat h_2\circ f_2$. Since there is a universal homogeneous retraction from $\bU$ to $\bV$, from Theorem~\ref{uhretract} it follows that there exist $\bC\in\cC$, $g_1:\bB_1\injto\bC$, $g_2:\bB_2\injto\bC$, $\hat h:\bC\to\bV$, such that the following diagram commutes:
	\[
	\begin{psmatrix}
	& & [name=W] \bV\\
	[name=B1] \bB_1 & [name=C] \bC\\
	[name=A] \bA & [name=B2] \bB_2
	\ncline{H->}{A}{B1}<{f_1}
	\ncline{H->}{A}{B2}^{f_2}
	\ncline{H->}{B1}{C}^{g_1}
	\ncline{H->}{B2}{C}>{g_2}
	\ncarc[arcangle=25]{->}{B1}{W}^{\hat h_1}
	\ncarc[arcangle=-25]{->}{B2}{W}>{\hat h_2}
	\ncline{->}{C}{W}<{\hat h}
	\end{psmatrix}
	\]
	Let $h:= s\circ\hat h$. Then $h\circ g_i=s\circ\hat h\circ g_i = s\circ\hat h_i=s\circ t\circ h_i = h_i$, for $i\in\{1,2\}$. Hence the following diagram commutes:
	\[
	\begin{psmatrix}
	& & [name=W] \bW\\
	[name=B1] \bB_1 & [name=C] \bC\\
	[name=A] \bA & [name=B2] \bB_2
	\ncline{H->}{A}{B1}<{f_1}
	\ncline{H->}{A}{B2}^{f_2}
	\ncline{H->}{B1}{C}^{g_1}
	\ncline{H->}{B2}{C}>{g_2}
	\ncarc[arcangle=25]{->}{B1}{W}^{h_1}
	\ncarc[arcangle=-25]{->}{B2}{W}>{h_2}
	\ncline{->}{C}{W}<{h}
	\end{psmatrix}
	\]
	We conclude that condition \eqref{cond1} of Theorem~\ref{uhretract} is fulfilled for $\bW$. 
	
	Let now $\bA,\bB\in\cC$, $\iota:\bA\injto\bB$, $h:\bA\to\bW$. Then $t\circ h:\bA\to\bV$. Hence, by condition \eqref{cond2} of Theorem~\ref{uhretract}, there exists $\tilde h:\bB\to\bV$, such that $t\circ h=\tilde h\circ\iota$. Define $\hat h:= s\circ\tilde h$. Then $\hat h\circ\iota = s\circ\tilde h\circ\iota = s\circ t\circ h = h$. Hence, condition \eqref{cond2} of Theorem~\ref{uhretract} is fulfilled for $\bW$. 
	
	Now, from Theorem~\ref{uhretract} it follows that there exists a universal homogeneous retract $\hat s$ from $\bU$ to $\bW$. 
\end{proof}

Of course, every structure is a retract of itself. If this trivial retract is induced by a universal homogeneous retraction, then this retraction is a universal homogeneous endomorphism. By Proposition~\ref{subretract}, we conclude that a homogeneous structure has the property that each of its retracts is induced by a universal homogeneous retraction if and only if it has a universal homogeneous endomorphism. In the following we will characterize all homogeneous structures that have a universal homogeneous endomorphism. Before formulating the criterion, we need to make a few preparations:
\begin{definition}
	Let $\cC$ be an age. We say that $\cC$ has the \emph{homo amalgamation property} (HAP) if for every $\bA$, $\bB_1$, $\bB_2$ from $\cC$, for all embeddings $f_1:\bA\injto\bB_1$, and for all homomorphisms $f_2:\bA\to\bB_2$ there exists some $\bC\in\cC$, an embedding $g_2:\bB_2\injto\bC$ and a homomorphism $g_2:\bB_2\to\bC$ such that the following diagram commutes:
	\[
	\begin{psmatrix}
	[name=B1]\bB_1 & [name=C]\bC\\
	[name=A]\bA & [name=B2]\bB_2
	\ncline{H->}{A}{B1}<{f_1}
	\ncline{->}{A}{B2}_{f_2}
	\ncline{H->}{B2}{C}>{g_2}
	\ncline{->}{B1}{C}^{g_1}
	\end{psmatrix}
	\]
\end{definition}
It was proved by Dolinka \cite[Prop.3.8]{Dol11} that a countable homogeneous structure is homomorphism homogeneous if and only if its age has the HAP. Recall that a structure is \emph{homomorphism homogeneous} if every homomorphism between finitely generated substructures can be extended to an endomorphism of the structure. Moreover, a structure $\bU$ is called \emph{weakly homomorphism homogeneous} if for all $\bA,\bB\in\Age(\bU)$, for all homomorphisms $h:\bA\to\bU$, and for all embeddings $\iota:\bA\injto\bB$, there exists a homomorphism $\hat{h}:\bB\to\bU$ such that the following diagram commutes:
\[
\begin{psmatrix}
	[name=B] \bB & [name=U] \bU\\
	[name=A] \bA 
	\ncline{H->}{A}{B}<{\iota}
	\ncline{->}{A}{U}>{h}
	\ncline{->}{B}{U}^{\hat{h}}
\end{psmatrix}
\] 
Clearly, homomorphism homogeneity implies weak homomorphism homogeneity, and for countable structures, these notions are equivalent. 

\begin{proposition}\label{hapext}
	Let $\cC$ be a \Fraisse-class. Let $\bU$ be its \Fraisse-limit. Then $\bU$ has a universal homogeneous endomorphism if and only if $\cC$ has the amalgamated extension property and the HAP.
\end{proposition}
\begin{proof}
	Let $\fA=(\overline{\cC},\injto)$, $\fC=(\overline{\cC},\to)$, and let $\fB$ be the subcategory of $\fC$ that consists only of the object $\bU$ and the identity morphism $1_\bU$. Let further $F:\fA\to\fC$ and $G:\fB\to\fC$ be the identical embeddings. Then conditions (1)--(5) of Theorem~\ref{mainconstruction} are clearly fulfilled.
	
	``$\Leftarrow$'' First we show that $(F\downarrow G)$ has the $(F\restr_{\fA_{<\aleph_0}}, G\restr_{\fB_{<\aleph_0}})$-amalgamation property: Let $\bA,\bB_1,\bB_2\in\fA_{<\aleph_0}$, $f_1:\bA\injto\bB_1$, $f_2:\bA\injto\bB_2$, $h_1:\bB_1\to \bU$, $h_2:\bB_2\to\bU$ such that $h_1\circ f_1 = h_2\circ f_2$. Since $\bB_1$ and $\bB_2$ are finitely generated, there exists a finitely generated substructure $\bU'$ of $\bU$, and homomorphisms $h_1':\bB_1\to\bU'$, $h_2':\bB_2\to\bU'$ such that $\iota\circ h'_1=h_1$ and $\iota\circ h'_2=h_2$ where $\iota:\bU'\to\bU$ is the identical embedding.  Since $\cC$ has the amalgamated extension property,  there exists $\bC\in\cC$, $g_1:\bB_1\injto\bC$, $g_2:\bB_2\injto\bC$, $\bU''\in\cC$, $h:\bC\to\bU''$, $k:\bU'\injto\bU''$ such that the following diagram commutes: 
	\[
	\begin{psmatrix}
	& & & [name=GTp] \bU''\\
	& & [name=GT] \bU'\\
	[name=FB1] \bB_1 & [name=FC] \bC\\
	[name=FA] \bA & [name=FB2] \bB_2
	\ncline{H->}{FA}{FB1}<{f_1}
	\ncline{H->}{FA}{FB2}^{f_2}
	\ncline{H->}{FB1}{FC}^{g_1}
	\ncline{H->}{FB2}{FC}>{g_2}
	\ncarc[arcangle=25]{->}{FB1}{GT}^{h'_1}
	\ncarc[arcangle=-25]{->}{FB2}{GT}>{h'_2}
	\ncarc[arcangle=30]{->}{FC}{GTp}^{h}
	\ncline{H->}{GT}{GTp}_{k}
	\end{psmatrix}
	\]
	Since $\bU$ is homogeneous, it is in particular weakly homogeneous and thus $\fA_{<\aleph_0}$-saturated. So there exists $\hat{\iota}:\bU''\to\bU$ such that $\hat{\iota}\circ k =\iota$. 
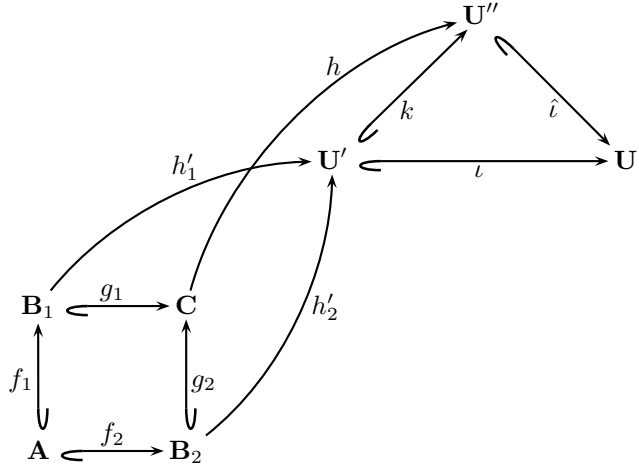
	\[
	\begin{psmatrix}
	& & & [name=Upp] \bU''\\
	& & [name=Up] \bU' & & [name=U]\bU\\
	[name=B1] \bB_1 & [name=C] \bC\\
	[name=A] \bA & [name=B2] \bB_2
	\ncline{H->}{A}{B1}<{f_1}
	\ncline{H->}{A}{B2}^{f_2}
	\ncline{H->}{B1}{C}^{g_1}
	\ncline{H->}{B2}{C}>{g_2}
	\ncarc[arcangle=25]{->}{B1}{Up}^{h'_1}
	\ncarc[arcangle=-25]{->}{B2}{Up}>{h'_2}
	\ncarc[arcangle=30]{->}{C}{Upp}^{h}
	\ncline{H->}{Up}{Upp}_{k}
	\ncline{H->}{Up}{U}_{\iota}
	\ncline{H->}{Upp}{U}_{\hat{\iota}}
	\end{psmatrix}
	\]
	It follows that also  the following diagram commutes:
	
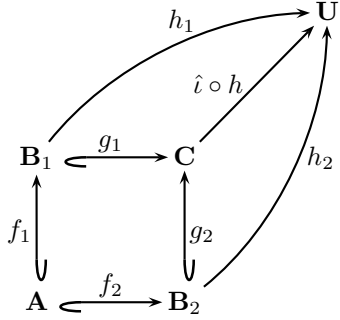
\[
	\begin{psmatrix}
	& & [name=U] \bU\\
	[name=B1] \bB_1 & [name=C] \bC\\
	[name=A] \bA & [name=B2] \bB_2
	\ncline{H->}{A}{B1}<{f_1}
	\ncline{H->}{A}{B2}^{f_2}
	\ncline{H->}{B1}{C}^{g_1}
	\ncline{H->}{B2}{C}>{g_2}
	\ncarc[arcangle=25]{->}{B1}{U}^{h_1}
	\ncarc[arcangle=-25]{->}{B2}{U}>{h_2}
	\ncline{->}{C}{U}<{\hat{\iota}\circ h}
	\end{psmatrix}
	\]
	We obtain that $\fA_{<\aleph_0}$ has the $(F\restr_{\fA_{<\aleph_0}},G\restr_{\fB_{<\aleph_0}})$-amalgamation property.
	
	Since $\bU$ contains a smallest substructure (namely, the substructure generated by the empty set), it follows that $\fA_{<\aleph_0}$ has an initial object. Moreover, $F$ maps this initial object to an initial object of $\fC$. Hence, from the $(F\restr_{\fA_{<\aleph_0}},G\restr_{\fB_{<\aleph_0}})$-amalgamation property of $\fA_{<\aleph_0}$ follows the  $(F\restr_{\fA_{<\aleph_0}},G\restr_{\fB_{<\aleph_0}})$-joint embedding property. Now, by Theorem~\ref{mainconstruction}, it follows that $(F\downarrow G)$ has an $(F\downarrow G)$-universal, $(F\downarrow G)_{<\aleph_0}$-homogeneous object $(\hat{\bU},i,\bU)$. It remains to show that $\hat{\bU}\cong\bU$. Clearly, $\Age(\hat{\bU})=\Age(\bU)$.  Since $\cC$ has the HAP, it follows from a result of Dolinka \cite[Prop.3.8]{Dol11}, that $\bU$ is homomorphism-homogeneous. In particular, it is weakly homomorphism-homogeneous. From this it follows directly, that $F$ and $G$ have the mixed amalgamation property. Finally, from Proposition~\ref{AFsat}, it follows that $\hat\bU$ is $\fA_{<\aleph_0}$-saturated. In other words, $\hat{\bU}$ is homogeneous. It follows from \Fraisse's theorem that $\hat{\bU}$ and $\bU$ are isomorphic. So without loss of generality we may assume that $\hat\bU=\bU$. Thus $u$ is a universal homogeneous endomorphism of $\bU$. 
	
	``$\Rightarrow$'' Suppose that $u$ is a universal homogeneous endomorphism of $\bU$. Then $(\bU,u,\bU)$ is an $(F\downarrow G)$-universal and $(F\downarrow G)_{<\aleph_0}$ homogeneous object in $(F\downarrow G)$. It follows from Theorem~\ref{mainconstruction} that $\fA$ has the $(F\restr_{\fA_{<\aleph_0}},G\restr_{\fB_{<\aleph_0}})$-amalgamation property, and the $(F\restr_{\fA_{<\aleph_0}},G\restr_{\fB_{<\aleph_0}})$-joint embedding property. Since $\bU$ is homogeneous, it is in particular weakly homogenous and thus $\fA_{<\aleph_0}$-saturated. Hence it follows from Proposition~\ref{AFsat} that $F$ and $G$ have the mixed amalgamation property. In other words, $\bU$ is weakly homomorphism homogeneous, and thus homomorphism homogeneous. It follows that $\cC$ has the HAP. 
	
	Let $\bA,\bB_1,\bB_2,\bU'\in\cC$, $f_1:\bA\injto\bB_1$, $f_2:\bA\injto\bB_2$, $h_1:\bB_1\to\bU'$, $h_2:\bB_2\to\bU'$, such that the following diagram commutes:
	
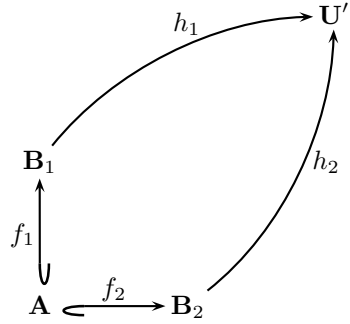
\[
	\begin{psmatrix}
	& & [name=Up] \bU'\\
	[name=B1] \bB_1 \\
	[name=A] \bA & [name=B2] \bB_2
	\ncline{H->}{A}{B1}<{f_1}
	\ncline{H->}{A}{B2}^{f_2}
	\ncarc[arcangle=25]{->}{B1}{Up}^{h_1}
	\ncarc[arcangle=-25]{->}{B2}{Up}>{h_2}
	\end{psmatrix}
	\]
	Since $\bU$ is universal for $\overline{\cC}$, there exists an embedding $\iota:\bU'\injto\bU$. By the $(F\restr_{\fA_{<\aleph_0}},G\restr_{\fB_{<\aleph_0}})$-amalgamation property, there exists $\bC\in\cC$, $g_1:\bB_1\injto\bC$, $g_2:\bB_2\injto\bC$, $h:\bC\to\bU$ such that the following diagram commutes:
	
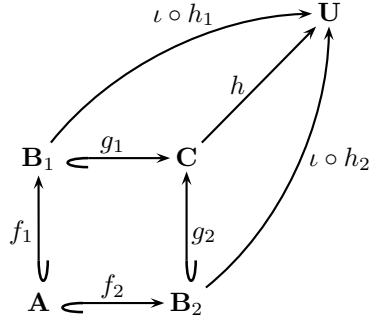
\[
	\begin{psmatrix}
	& & [name=U] \bU\\
	[name=B1] \bB_1 & [name=C] \bC\\
	[name=A] \bA & [name=B2] \bB_2
	\ncline{H->}{A}{B1}<{f_1}
	\ncline{H->}{A}{B2}^{f_2}
	\ncline{H->}{B1}{C}^{g_1}
	\ncline{H->}{B2}{C}>{g_2}
	\ncarc[arcangle=25]{->}{B1}{U}^{\iota\circ h_1}
	\ncarc[arcangle=-25]{->}{B2}{U}>{\iota\circ h_2}
	\ncline{->}{C}{U}<{h}
	\end{psmatrix}
	\]
	Since $\bC$ and $\bU'$ are finitely generated, there exists a finitely generated substructure $\bU''$ of $\bU$ with identical embedding $\iota'':\bU''\injto\bU$, such that both, $h$ and $\iota$ factor through $\iota''$ --- i.e. there exists $k:\bU'\injto\bU''$, $h':\bC\to\bU''$ such that $h=\iota''\circ h'$ and $\iota=\iota''\circ k$. In particular, the following diagram commutes:
	
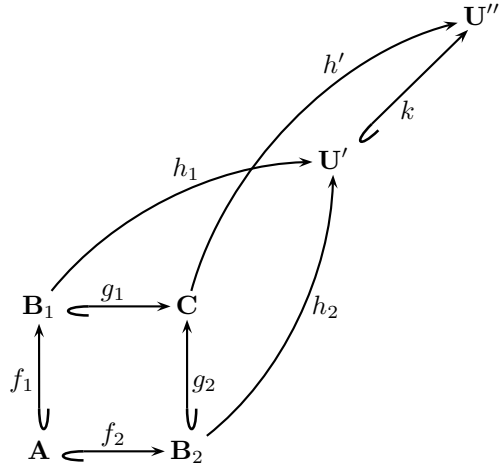
\[
	\begin{psmatrix}
	& & & [name=GTp] \bU''\\
	& & [name=GT] \bU'\\
	[name=FB1] \bB_1 & [name=FC] \bC\\
	[name=FA] \bA & [name=FB2] \bB_2
	\ncline{H->}{FA}{FB1}<{f_1}
	\ncline{H->}{FA}{FB2}^{f_2}
	\ncline{H->}{FB1}{FC}^{g_1}
	\ncline{H->}{FB2}{FC}>{g_2}
	\ncarc[arcangle=25]{->}{FB1}{GT}^{h_1}
	\ncarc[arcangle=-25]{->}{FB2}{GT}>{h_2}
	\ncarc[arcangle=30]{->}{FC}{GTp}^{h'}
	\ncline{H->}{GT}{GTp}_{k}
	\end{psmatrix}
	\]
	It follows that $\cC$ has the amalgamated extension property.
\end{proof}

\begin{proposition}
	Let $\bU$ be a countable structure that has a universal homogeneous endomorphism. Then $\bU$ is homogeneous if and only if $\bU$ is homomorphism homogeneous.
\end{proposition}
\begin{proof}
	``$\Rightarrow$'' this follows directly from Proposition~\ref{hapext}.
	
	``$\Leftarrow$'' Let $\cC:=\Age(\bU)$. Take the categories $\fA$, $\fB$, $\fC$ and the functors $F$ and $G$ as given in the proof of Proposition~\ref{hapext}. Clearly, Conditions (1),\dots,(6) of Proposition~\ref{AFsat} are fulfilled.

	Since $\bU$ has a universal homogeneous endomorphism, it follows, that  $(\bU,u,\bU)$ is an $(F\downarrow G)$-universal, $(F\downarrow G)_{<\aleph_0}$-homogeneous object in $(F\downarrow G)$. Hence, by Theorem~\ref{DroGoe}, Condition (7) of Proposition~\ref{AFsat} is fulfilled, too. 
	
	Since $\bU$ is homomorphism homogeneous, it follows that it is weakly homomorphism homogeneous. This means that $F$ and $G$ have the mixed amalgamation property.
	Now we can use Proposition~\ref{AFsat} to conclude that $\bU$ is homogeneous.
\end{proof}

\begin{corollary}\label{allretsunivhom}
	Let $\cC$ be a \Fraisse-class with \Fraisse-limit $\bU$. Then $\cC$ has the HAP and the amalgamated extension property if and only if every retract of $\bU$ is induced by a universal homogeneous retraction. 
\end{corollary}
\begin{proof}
``$\Rightarrow$'' From Proposition~\ref{hapext} it follows that $\bU$ has a universal homogeneous endomorphisms $u$. Since $u$, is a retraction, it follows from Lemma~\ref{subretract} that every retract of $\bU$ is induced by a universal homogeneous retraction.

``$\Leftarrow$'' Since every retract of $\bU$ is induced by a universal homogeneous retraction, it follows that in particular the trivial retract $\bU$ of $\bU$ is induced by a universal homogeneous retraction. In other words, $\bU$ has a universal homogeneous endomorphism. From Proposition~\ref{hapext}, it follows that $\cC$ has the amalgamated extension property and the HAP.
\end{proof}

\begin{example}\label{exretracts}
	Homogeneous structures that have both, the HAP and the amalgamated extension property, include:
	\begin{itemize}
		\item the Rado graph $R$ ($\Age(R)$ is the class of all finite simple graphs),
		\item the countable generic poset $\mathbb{P}=(P,\le)$ ($\Age(\mathbb{P})$ is the class of all finite posets),
		\item the countable atomless Boolean algebra $\mathbb{B}$ ($\Age(\mathbb{B})$ is the class of finite Boolean algebras),
		\item the countable universal homogeneous semilattice $\Omega$ ($\Age(\Omega)$  is the class of all finite semilattices),
		\item the countable universal homogeneous distributive lattice $\mathbb{D}$ ($\Age(\mathbb{D})$ is the class of all finite distributive lattices),
		\item the infinite-dimensional vector-space $\mathbb{F}^\omega$ for any countable field $\mathbb{F}$ ($\Age(\mathbb{F}^\omega)$ is the class of all finite-dimensional $\mathbb{F}$-vector spaces), 
		\item the rationals $(\mathbb{Q},\le)$ (the age of $\mathbb{Q}$ is the class of all finite chains),
		\item the rational Urysohn space $\mathbb{U}_{\mathbb{Q}}$ (the age of $\mathbb{U}_{\mathbb{Q}}$ is the class of all  finite metric spaces with rational distances),
	\item the rational Urysohn sphere of radius $1$ (its age is the class of all finite metric spaces with rational distances $\le 1$). 
	\end{itemize}
	In particular, each of these structures has a universal homogeneous endomorphism. Moreover, by Corollary~\ref{allretsunivhom}, all retracts are induced by universal homogeneous retractions.
\end{example}

\begin{remark}
	A universal homogeneous endomorphism $u$ of  $(\mathbb{Q},\le)$ has to have a few peculiar properties.  From the universality, it follows that  the preimage of each point is a dense convex subset of $\mathbb{Q}$. From homogeneity, it follows that  the preimage of every point has neither a greatest, nor a smallest element. Hence, every $x\in\mathbb{Q}$ is contained in an open interval on which $u$ is constant. However, $u$ itself can not be constant, since $u$ also has to be surjective. Since the existence of a function with these properties on the first sight is contra-intuitive, let us sketch a construction. We will proceed similarly to the construction of the Devil's staircase (a.k.a. the Cantor function). Let $A$ be the set of all rationals in the interval $(0,1)$ that do not have a finite expansion to the basis $3$, and whose expansion to basis $3$ contains at least once the digit $1$. Let $B$ be the set of all rationals from the interval $(0,1)$ that have a finite expansion in  base $2$. For $a\in A$ we obtain $u(a)$ by the following process (cf. \cite{DovMarRyaVuo06}):
	\begin{enumerate}
		\item represent $a$ in base $3$,
		\item replace all digits after the first digit $1$ by $0$,
		\item in the resulting tuple, replace all digits $2$ by $1$,
		\item interpret the resulting tuple in base $2$. Thus is defined $u(a)$.
	\end{enumerate}
	Clearly, $u(a)\in B$. Also, it is not hard to see, that $u$ is monotonous and surjective. Both $(A,\le)$ and $(B,\le)$ are dense unbounded chains. Thus, by Cantor's theorem, both are isomorphic to $(\mathbb{Q},\le)$. In the following, we will identify $A$ and $B$ with $\mathbb{Q}$ (using arbitrary isomorphisms). Note now, that the full preimage of every point from $B$ is a dense unbounded chain, again, and hence it is isomorphic to $\mathbb{Q}$. By the monotonicity and surjectivity of $u$, it is an open interval in $(\mathbb{Q},\le)$. Altogether, $u$ has the desired properties.
\end{remark}

\begin{remark}
	Kubi\'s in  \cite{Kub13}, among other things,  characterizes the retracts of homogeneous structures whose age have the HAP, and the amalgamated extension property. The results of this section imply Kubi\'s' result. However, additionally  we showed that for this class of structures every retract is induced by a universal homogeneous retraction and that no other homogeneous structure has this property. Moreover, we characterized the retracts that are induced by universal homogeneous retractions in all homogeneous structures. In this sense, our results nicely complement Kubi\'s's findings. 
\end{remark}

\section{Generating Polymorphism clones of homogeneous structures}\label{genclones}
 Let us in the beginning recall the definition of a clone: Let $A$ be any set, then by $O_A^{(n)}$ we denote the set of all functions from $A^n$ to $A$. Further we define $O_A:=\bigcup_{n\in\bN\setminus\{0\}} O_A^{(n)}$. Finally, for a set $F\subseteq O_A$, by $F^{(n)}$ we will denote $F\cap O_A^{(n)}$.

We distinguish special functions --- the projections $e_i^n$ --- that act like
\[ e_i^n: (x_1,\dots,x_n)\mapsto x_i.\]
The set of all projections on $A$ is denoted by $J_A$.

For $f\in O_A^{(n)}$, $g_1,\dots,g_n\in O_A^{(m)}$ we define the function $f\circ \langle g_1,\dots,g_n\rangle\in O_A^{(m)}$ according to 
\[ f\circ\langle g_1,\dots,g_n\rangle : (x_1,\dots, x_m)\mapsto f(g_1(x_1,\dots,x_m),\dots, g_n(x_1,\dots,x_m)).\]
\begin{definition}	
	 A subset $C\subseteq O_A$ is called a \emph{clone} on $A$ if $J_A\subseteq C$, and if whenever $f\in C^{(n)}$, $g_1,\dots,g_n\in C^{(m)}$ then
	$f\circ \langle g_1,\dots,g_n\rangle\in C^{(m)}$ ($n,m\in\bN\setminus\{0\}$).
\end{definition}
Clearly, the set of clones on a given set $A$ forms a complete algebraic lattice with the intersection as infimum. Therefore, it makes sense to talk about generating sets of clones. If $M\subseteq O_A$, then by $\langle M\rangle_{O_A}$ we will denote the smallest clone on $A$ that contains all functions from $M$. We also say that $M$ is a \emph{generating set} of $\langle M\rangle_{O_A}$. 

Let now $\bA$ be a structure. An $n$-ary polymorphism of $\bA$ is a homomorphism from $\bA^n$ to $\bA$. The set of polymorphisms of $\bA$ of every arity is denoted by $\Pol(\bA)$. It is not hard to see that $\Pol(\bA)$ forms a clone on $A$. The unary polymorphisms of $\bA$ are the endomorphisms of $\bA$ (denoted by $\End\bA$). The submonoid of $\End\bA$ that consists of all homomorphic self-embeddings of $\bA$ will be denoted by $\Emb\bA$.
 
The clone of polymorphisms is a structural invariant, related to the group of automorphisms and the monoid of endomorphisms. In this section we will derive results about generating sets of the polymorphism clones of homogeneous structures. The initial motivation is an old result by Sierpi\'nski who showed in \cite{Sie45} that for every set $A$, the clone $O_A$ of all functions on $A$ is generated by the set of all binary functions on $A$. 

Another motivation comes from the paper \cite{HowRusHig98} where (among other things) it is shown that the semigroup of all transformations on an infinite set $A$ is generated by the set of permutations of $A$ and two additional functions. Using Ru\v{s}kuc' notion of relative ranks (cf. \cite[Def.1.1]{Rus94}) one can say that the semigroup of transformations of $A$ has rank $2$ modulo the full symmetric group on $A$. 

We will adopt the notion of relative rank to the world of clones. Let $F$ be a clone on a set $A$ and let $M\subseteq F$ be an arbitrary subset of $F$. A subset $N$ of $F$ is called \emph{generating set of $F$ modulo $M$} if $\langle M\cup N\rangle_{O_A}=F$. The relative rank of $F$ modulo $M$ will be the smallest cardinal of a generating set $N$ of $F$ modulo $M$. It will be denoted by $\rank(F:M)$.

\begin{proposition}\label{PolGen}
	Let $\bA$ be a structure that such that there exists a retraction $r:\bA\epito \bA^2$. Then
	\[\rank(\Pol\bA: \End\bA) = 1.\]
	In particular, $\Pol \bA$ is generated by $\End\bA$ together with one additional section $\epsilon:\bA^2\injto\bA$.
\end{proposition}
\begin{proof}
	Let $\epsilon:\bA^2\injto\bA$ be an embedding, such that $r\circ\epsilon= 1_{\bA^2}$.
	
	We define $\epsilon_1:=\epsilon$, $r_1:=r$, 
	\begin{align*}
		\epsilon_{i+1} &: \bA^{i+2}\injto \bA & \epsilon_{i+1}&:= \epsilon\circ(\epsilon_i\times 1_{\bA})&\text{and}\\
		r_{i+1}&:\bA\epito\bA^{i+2} & r_{i+1}&:= (r_i\times 1_{\bA})\circ r.
	\end{align*}
	
	\textbf{Claim 1:} For all $i\in\bN\setminus\{0\}$ we have $\epsilon_i\in\langle\End\bA\cup\{\epsilon\}\rangle_{O_A}$.
	The proof proceeds by induction on $i$. For $i=1$, nothing needs to be proved. Suppose, $\epsilon_j\in \langle\End\bA\cup\{\epsilon\}\rangle_{O_A}$. It can easily be checked that $\epsilon_{j+1}=\epsilon\circ\langle \epsilon_j\circ\langle e_1^{j+2},\dots, e_{j+1}^{j+2}\rangle, e_{j+2}^{j+2}\rangle$. 
	Thus, Claim 1 is proved.
	
	\textbf{Claim 2:} For all $i\in\bN\setminus\{0\}$, we have $r_i\circ \epsilon_i = 1_{\bA^{i+1}}$. Again we proceed by induction. For $i=1$ nothing needs to be proved. Suppose, $r_j\circ\epsilon_j = 1_{\bA^{j+1}}$. We compute
	\begin{align*}
		r_{j+1}\circ \epsilon_{j+1}(x_1,\dots,x_{j+2}) &= \big((r_i\times 1_{\bA})\circ r\big)\circ\big(\epsilon\circ(\epsilon_j\times 1_{\bA})\big)(x_1,\dots,x_{j+2})\\
		&= \big((r_j\times 1_{\bA})\circ r\big)\Big(\epsilon\big(\epsilon_j(x_1,\dots,x_{j+1}),x_{j+2}\big)\Big)\\
		&= (r_j\times 1_{\bA})\big(\epsilon_j(x_1\dots,x_{j+1}),x_{j+2}\big)\\
		&= (x_1,\dots,x_{j+2}). 
	\end{align*}
	Thus, Claim 2 is proved.
		
	We will now prove for every $i\ge 2$, that  $\Pol^{(i+1)}\bA\subseteq\langle\End\bA\cup\{\epsilon\}\rangle_{O_A}$.  Take any $f\in\Pol^{(i+1)}\bA$. Then $f\circ r_{i}\in\End\bA$.
	But then, by Claim 2, we have $f=(f\circ r_i)\circ\epsilon_i$. Now, from Claim 1 follows that $f\in\langle\End\bA\cup\{\epsilon\}\rangle_{O_A}$. 
	
	It remains to show that $\Pol\bA\neq\langle \End\bA\rangle_{O_A}$. However, since $\epsilon$ is a section, it has to depend on both of its variables, while  $\langle \End\bA\rangle_{O_A}$ consists only of essentially unary functions.
\end{proof}

\begin{corollary}\label{CorPolGen}
	With the notions from above, if $\bA$ has a universal endomorphism $u$, then 
	\[\rank(\Pol\bA: \Emb\bA) \le 2.\]
	In particular, $\Pol \bA$ is generated by $\Emb\bA$ together with $u$ and one additional section $\epsilon:\bA^2\injto\bA$.
\end{corollary}
\begin{proof}
	From the definition of universality it follows immediately that $\End\bA$ is generated by $\Emb\bA$ together with $u$. With this observation, the claim follows directly from  Proposition~\ref{PolGen}.
\end{proof}

Now we can add up our previous observations to obtain:
\begin{theorem}\label{clgen}
	Let $\cC$ be a \Fraisse-class with \Fraisse-limit $\bU$, such that
	\begin{enumerate}
		\item $\cC$ is closed with respect to finite products,
		\item $\cC$ has the HAP,
		\item $\cC$ has the amalgamated extension property.
	\end{enumerate}
	Then $\rank(\Pol\bU:\Emb\bU)\le 2$. In particular, $\Pol\bU$ is generated by $\Emb\bU$ together with unary and a binary polymorphism.
\end{theorem}
\begin{proof}
	Since $\cC$ is closed with respect to finite products, it follows that $\Age(\bU^2)\subseteq\cC$. 
		
	Since $\cC$ has the HAP, it follows that $\bU$ is homomorphism homogeneous and thus also weakly homomorphism homogeneous. Let now $\bA\in\cC$, $h:\bA\to\bU^2$, $\iota:\bA\injto\bB$. Then we can write $h=\langle h_1,h_2\rangle$, where $h_1,h_2:\bA\to\bU$. Since $\bU$ is weakly homomorphism homogeneous, there exist $\hat{h}_1,\hat{h}_2:\bB\to\bU$, such that $\hat{h}_1\circ\iota=h_1$ and $\hat{h}_2\circ\iota=h_2$. With $\hat{h}:=\langle \hat{h}_1,\hat{h}_2\rangle$ this amounts to saying $h=\hat{h}\circ \iota$. Thus, condition Condition~\eqref{cond2} of Theorem~\ref{uhretract} is fulfilled.
	
	Since $\cC$ has the amalgamated extension property, by Lemma~\ref{amalgextred} it follows that also Condition~\eqref{cond1} of Theorem~\ref{uhretract} is fulfilled for $\bT=\bU^2$. 
	Thus, by Theorem~\ref{uhretract}, there exists a universal homogeneous retraction $r$ from $\bU$ to $\bU^2$. Let $\epsilon:\bU^2\injto\bU$ be a section for $r$ 
	
	Using again that $\cC$ has the HAP and the amalgamated extension property, from Proposition~\ref{hapext} it follows that $\bU$ has a universal homogeneous endomorphism $u$. Thus, by Corollary~\ref{CorPolGen}, $\Pol\bU$ is generated by $\Emb\bU$ together with $u$ and $\epsilon$. 
	
	Thus $\rank(\Pol\bU:\Emb\bU)\le 2$. 
\end{proof}

\begin{example}\label{expoly}
	The ages of all the structures from Example~\ref{exretracts} have the HAP and the amalgamated extension property. The only structure from there whose age is not closed with respect to finite products is $(\bQ,\le)$. Hence the polymorphism clones of the following structures  all have relative rank $\le 2$ modulo the respective self-embedding monoids: 
	\begin{itemize}
		\item the Rado graph $R$,
		\item the countable generic poset $\mathbb{P}=(P,\le)$,
		\item the countable atomless Boolean algebra $\mathbb{B}$,
		\item the countable universal homogeneous semilattice $\Omega$,
		\item the countable universal homogeneous distributive lattice $\mathbb{D}$,
		\item the infinite-dimensional vector-space $\mathbb{F}^\omega$ for any countable field $\mathbb{F}$, 
		\item the rational Urysohn space $\mathbb{U}_{\mathbb{Q}}$,
		\item the rational Urysohn sphere of radius $1$. 
	\end{itemize}
\end{example}

\section{Cofinality of polymorphism clones of homogeneous structures}\label{cofinality}
The \emph{cofinality} of a non-finitely generated algebraic structure $\mathbb{A}$ is the least cardinal $\lambda$ for which there exists an  increasing chain $(\mathbb{A}_i)_{i<\lambda}$ of proper substructures of $\mathbb{A}$ with the property that $\mathbb{A}=\bigcup_{i<\lambda}\mathbb{A}_i$. It is denoted by $\cf(\mathbb{A})$. 
This important invariant has been studied extensively for groups, but also for semigroups, boolean algebras, semilattices and other general algebraic structures. 

The goal of this section is to describe a class of homogeneous structures whose polymorphism clones have uncountable cofinality. 

First we observe that finitely generated clones have no cofinality. 
Moreover, non-finitely generated clones that are generated by a countably infinite set of functions always have cofinality $\aleph_0$. Thus the concept of cofinality becomes interesting only for very large clones. 

If $\uF$ is a clone on a set $A$, then we can define a sequence $(\uF_i)_{i<\aleph_0}$ of subclones of $\uF$ by $\uF_i:=\langle \uF^{(i)}\rangle_{O_A}$. If each of the $\uF_i$ is a proper subclone of $\uF$, then it follows that $\uF$ has countable cofinality. Thus, any clone with uncountable cofinality has to be generated by its $k$-ary part for some $k$.

The following proposition links the concept of cofinality for semigroups with the one for clones: 
\begin{proposition}\label{cfcond}
	Let $\uF$ be a clone on a set $A$, and let $\uS\subseteq \uF^{(1)}$ be a transformation semigroup. If $\cf(\uS)>\aleph_0$, and if $\uF$ has finite rank modulo $\uS$, then $\cf(\uF)>\aleph_0$, too. 
\end{proposition}
\begin{proof}
	Suppose that $\cf(\uF)=\aleph_0$ and let $(\uF_i)_{i\in\bN}$ be a cofinal chain for $\uF$.  Define $\uS_i:=\uS\cap \uF_i$ ($i\in\bN$).  Then $\uS_i$ is a transformation semigroup and we have $\bigcup_{i\in\bN} \uS_i =\uS$. Since $\uS$ has uncountable cofinality, there has to exist some $i_0\in\bN$, such that $\uS_{i_0}=\uS$. 
	
	Since $\uF$ has a finite rank modulo $\uS$, there exists a finite set $H\subset\uF$, such that $\uS\cup H$ generates $\uF$. Since $(\uF_i)_{i\in\bN}$ is cofinal, it follows that there exists some $j_0\in\bN$ such that $H\subseteq \uF_{j_0}$. Let $k_0=\max\{i_0,j_0\}$. Then $\uS\subseteq \uF_{k_0}$ and $H\subseteq \uF_{k_0}$. Hence $\uF_{k_0}=\uF$ --- contradiction. It follows that $\uF$ has uncountable cofinality.    
\end{proof}

Dolinka, in \cite{Dol11}, gave sufficient conditions for the endomorphism monoid of a homogeneous structure to have uncountable cofinality. The result given there is stronger than required by us at this point because it proves uncountable \emph{strong} cofinality. We will meet this concept again later on. For now it is enough to know that uncountable strong cofinality implies uncountable cofinality.
\begin{theorem}{\cite[Thm.4.2]{Dol11},\cite[Lem.2.4]{MalMitRus09}}\label{dol}
	Let $\cC$ be a strict \Fraisse-class that has the (HAP). Let $\bU$ be the \Fraisse-limit of $\cC$. If the coproduct of $\aleph_0$ copies of $\bU$ exists in $(\overline{\cC},\to)$, 
	and if $\End\bU$ is not finitely generated, then $\End\bU$ has uncountable strong cofinality.
\end{theorem}

Combining this result with our observations from Section~\ref{genclones}, we obtain:
\begin{theorem}\label{clcf}
	Let $\cC$ be a strict \Fraisse-class with \Fraisse-limit $\bU$, such that
	\begin{enumerate}
		\item $\cC$ is closed with respect to finite products,
		\item $\cC$ has the HAP,
		\item the coproduct of $\aleph_0$ copies of $\bU$ exists in $(\overline{\cC},\to)$,
		\item $\End\bU$ is not finitely generated.  
	\end{enumerate}
	Then $\Pol\bU$ has uncountable cofinality.
\end{theorem}
\begin{proof}
	From Theorem~\ref{clgen} it follows that $\rank(\Pol\bU:\Emb\bU)$ is finite. Hence, also $\rank(\Pol\bU:\End\bU)$ is finite. From Theorem~\ref{dol}, it follows that $\End\bU$ has uncountable cofinality. Applying Proposition~\ref{cfcond}, we find that $\Pol\bU$ has uncountable cofinality. 
\end{proof}

\begin{example}\label{excf}
	It is easy to see that each of the structures from Example~\ref{expoly}, except for the rational Urysohn-space and the rational Urysohn-sphere, fulfills the conditions of Theorem~\ref{clcf}. Hence their polymorphism clones all have uncountable cofinality.  
	
	The clone $O_A$ for a countably infinite set $A$ has uncountable cofinality, too, since it is the polymorphism clone of the structure with base set $A$ over the empty signature, and since this structure fulfills the conditions of Theorem~\ref{clcf}.   
\end{example}

\section{The Bergman property for polymorphism clones of homogeneous structures}\label{sBergman}

Recall that a group $G$ is said to have the \emph{Bergman property} if for every generating set $H$ of $G$ there exists a natural number $k$ such that every other element of $G$ can be represented by a group-word of length at most $k$ in the generators from $H$. Another way to put this is that every connected Cayley graph of $G$ has a finite diameter. This concept originates from \cite{Ber06} where this property is proved for the symmetric groups on infinite sets.

Meanwhile the Bergman property has been studied intensively for groups and semigroups. 

One possibility to define the Bergman property for clones would be to bound the depth of a term needed for the generation of a given element from a given generating set of the clone. However,  such a naive definition leads to trouble since the minimum depth of a term that defines a given function may not only depend on the generating set but also on the arity of the function in question. For example,  by Sierpi\'nski's Theorem we know that the clone $O_A$ is generated by its binary part. However, every $k$-ary function on $A$ with $k$ non-fictitious variables needs a term of depth at least $\log_2(k)$ to be generated from binary functions. Therefore we prefer to give the following definition that takes into account the arity of functions:
\begin{definition}
	A clone $\uF$ is said to have the \emph{Bergman-property} if for every generating set $H$ of $\uF$ and every $k\in\bN\setminus\{0\}$ there exists some $n\in\bN$ such that every $k$-ary function from $F$ can be represented by a term of depth at most $n$ from the functions in $H$.
\end{definition}

Let $A$ be a set.  We define a complex product on $O_A$ according to 
\[ U_1\cdot U_2 = \{f\circ\langle g_1,\dots,g_n\rangle\mid f\in U_1^{(n)}, g_1,\dots,g_n\in U_2^{(m)},\, n,m\in\bN\setminus\{0\}\}.\]
\begin{lemma}
	The complex product defines a semigroup operation on the power set of $O_A$.\qed 
\end{lemma}
Let now $U\subseteq O_A$. For $k\in\bN\setminus\{0\}$, we define inductively:
\[ U^{[k,0]}:= J_A^{(k)},\qquad U^{[k,i+1]}:= U\cdot U^{[k,i]}\cup U^{[k,i]}\]
Note that if $\uF=\langle U\rangle_{O_A}$, then $U^{[k,i]}$ is the set of all functions from $\uF^{(k)}$ that can be generated from elements of $U$ by terms of depth $\le i$. Hence, a clone $\uF$ has the Bergman property if and only if for every generating set $H$ of $\uF$ end for every $k\in\bN\setminus\{0\}$ there exists some $n\in\bN$ such that $H^{[k,n]}=\uF^{(k)}$.

In general, for a clone $\uF=\langle H\rangle_{O_A}$, and for some $k\in\bN\setminus\{0\}$ we call $\uF$ \emph{$k$-Cayley-bounded with respect to $H$} if there exists some $n\in\bN$ such that $H^{[k,n]}=F^{(k)}$. Thus, $\uF$ has the Bergman property if and only if for every $k\in\bN\setminus\{0\}$ and for every generating set $H$ of $\uF$ we have that $\uF$ is $k$-Cayley-bounded with respect to $H$.
\begin{lemma}
	For every $i\in\bN$, $k\in\bN\setminus\{0\}$ we have $U^{[k,i]}\subseteq U^{[k,i+1]}$. \qed
\end{lemma}

\begin{lemma}\label{strong}
	For every $U\subseteq \uF$ and for all $i,j\in\bN$, $k,l\in\bN\setminus\{0\}$ we have $U^{[l,n]}\cdot U^{[k,m]}\subseteq U^{[k,n+m]}$.
\end{lemma}
\begin{proof}
	We proceed by induction on $n$: Clearly, $U^{[l,0]}\cdot U^{[k,m]}= J_A^{(l)}\cdot U^{[k,m]}= U^{[k,m]}$. 
	
	Suppose we have that $U^{[l,i]}\cdot U^{[k,m]}\subseteq U^{[k,i+m]}$ for some $i$. Then 
	\begin{align*}
		U^{[l,i+1]}\cdot U^{[k,m]} &= (U\cdot U^{[l,i]}\cup U^{[l,i]})\cdot U^{[k,m]}\\
		&= (U\cdot U^{[l,i]})\cdot U^{[k,m]}\cup U^{[l,i]}\cdot U^{[k,m]}\\
		&= U\cdot (U^{[l,i]}\cdot U^{[k,m]})\cup U^{[l,i]}\cdot U^{[k,m]}\\
		&\subseteq U\cdot U^{[k,i+m]}\cup U^{[k,i+m]}= U^{[k,i+1+m]}.
	\end{align*}
\end{proof}

According to \cite{DroGoe05}  the strong cofinality of a group $G$ is the least cardinal $\lambda$ such that there exists a chain $(G_i)_{i<\lambda}$ of proper subsets of $G$ whose union is equal to $G$ such that
\begin{enumerate}
	\item for all $i<\lambda$ we have $G_i=G_i^{-1}$,
	\item for all $i<\lambda$ there exists some $j$ with $i\le j<\lambda$ such that $G_i^2\subseteq G_j$.
\end{enumerate}
It was shown in \cite{DroGoe05} that a group has uncountable strong cofinality if and only if it has uncountable cofinality and the Bergman property. 

The concept of strong cofinality was defined for semigroups in \cite{MalMitRus09}: Namely, the strong cofinality of a semigroup $S$ is the least cardinal $\lambda$ such that there exists a chain $(S_i)_{i<\lambda}$ of proper subsets of $S$ whose union is equal to $S$ such that for every $i$ there exists some $j$ such $S_i^2\subseteq S_j$. In the same paper it was also shown that a semigroup has uncountable strong cofinality if and only if it has uncountable cofinality and the Bergman property. 

Following is a definition of the strong cofinality for clones:
\begin{definition}\label{scfdef}
	For a clone $\uF$ and a cardinal $\lambda$, a chain $(U_i)_{i<\lambda}$ of proper subsets of $\uF$ is called \emph{strong cofinal chain} of length $\lambda$ for $\uF$ if 
	\begin{enumerate}
		\item $\bigcup_{i<\lambda} U_i = \uF$,
		\item there exists a $k_0\in\bN\setminus\{0\}$ such that for all $i<\lambda$ and $k\in\bN\setminus\{0\}$ with $k\ge k_0$ holds $U_i^{(k)}\subsetneq \uF^{(k)}$,
		\item \label{pntthree}for all $i<\lambda$  there exists some $j<\lambda$ such that for all $k\in\bN\setminus\{0\}$ holds $U_i^{[k,2]}\subseteq U_j$. 
	\end{enumerate}
	The \emph{strong cofinality} of $\uF$ is the least cardinal $\lambda$ such that there exists a strong cofinal chain of length $\lambda$ for $\uF$.  It will be denoted by $\scf(\uF)$.
\end{definition}
Clearly, the strong cofinality of a clone $\uF$ is a regular cardinal that is less than or equal $\cf(\uF)$.

\begin{lemma}\label{strong2}
	Let $(U_i)_{i<\lambda}$ be a strong cofinal chain for $\uF$. Then for every $i<\lambda$,   and for every $r\in\bN$, there exists some $j<\lambda$ such that for all  $k\in\bN\setminus\{0\}$ holds $U_i^{[k,r]}\subseteq U_j$. 
\end{lemma}
\begin{proof}
	Indeed, for $r=0$ we calculate $U_i^{[k,0]}=J_A^{(k)}$, and we can take $j$ to be the smallest ordinal such that $J_A^{(k)}\subseteq U_j$ (such an ordinal exists by Definition~\ref{scfdef}(1)). 
	
	Suppose, the claim is shown for some $r$. 
	Let $m<\lambda$ be given such that $U_i^{[k,r]}\subseteq U_m$ for all $k$ and such that $i\le m$.  By definition we have $U_i^{[k,r+1]}=U_i\cdot U_i^{[k,r]}\cup U_i^{[k,r]}$. We calculate
	\[
		U_i\cdot U_i^{[k,r]}\subseteq U_i\cdot U_m^{(k)}\subseteq U_m\cdot U_m^{(k)}\subseteq U_m^{[k,2]}\subseteq U_l,
	\] 
	for some $l$ that fulfills $i\le m\le l<\lambda$ and for all $k$. Hence, $U_i^{[k,r+1]}\subseteq U_l$, for all $k$.  
\end{proof}
\begin{proposition}\label{bergman} 
	Let $\uF$ be a non-finitely generated clone on a set $A$. Then $\scf(\uF)>\aleph_0$ if and only if $\uF$ has the Bergman-property and $\cf(\uF)>\aleph_0$. 
\end{proposition}
\begin{proof} 
	``$\Rightarrow$'' Since $\scf(\uF)>\aleph_0$, we also have $\cf(\uF)>\aleph_0$. Let $U$ be a generating set of $\uF$, and suppose that $\uF$ is not $k$-Cayley-bounded with respect to $U$ --- i.e. for all $i\in\bN$ we have $U^{[k,i]}\subsetneq \uF^{(k)}$. Then we also have that for all $l>k$ and for all $i\in\bN$ we have  $U^{[l,i]}\subsetneq \uF^{(l)}$, for if $U^{[l,j]}= \uF^{(l)}$ for some $l>k$ and some $j$, then by Lemma~\ref{strong}, we have $\uF^{(k)}= \uF^{(l)}\cdot J_A^{(k)} = U^{[l,j]}\cdot U^{[k,0]}\subseteq U^{[k,j]}$ --- contradiction.
	
	We define $U_i:=\bigcup_{k\in\bN\setminus\{0\}} U^{[k,i]}$. Then $(U_i)_{i\in\bN}$ is a chain of proper subsets of $\uF$. Since $U$ is a generating set of $\uF$, we also have $\bigcup_{i\in\bN} U_i =\uF$. 
	By Lemma~\ref{strong}, we have 
	\begin{align*}
		U_i^{[k,1]}& = U_i\cdot J_A^{(k)}\cup J_A^{(k)}\\
		 &=\left(\bigcup_{l\in\bN\setminus\{0\}} U^{[l,i]}\right)\cdot U^{[k,0]}\cup J_A^{(k)}\\
		&\subseteq U^{[k,i]}\cup J_A^{(k)}=U^{[k,i]} 
	\end{align*}  
	Hence, again using Lemma~\ref{strong},
	\begin{align*}
		U_i^{[k,2]}&= U_i\cdot U_i^{[k,1]}\cup U_i^{[k,1]}\\
		&\subseteq U_i\cdot U^{[k,i]}\cup U^{[k,i]}\\
		&=\left(\bigcup_{l\in\bN\setminus\{0\}} U^{[l,i]}\right)\cdot U^{[k,i]}\cup U^{[k,i]}\\
		&\subseteq U^{[k,2i]}\cup U^{[k,i]}\\
		&= U^{[k,2i]}\subseteq U_{2i} 
	\end{align*}
	Hence $(U_i)_{i\in\bN}$ is a strong cofinal chain of length $\aleph_0$ for $\uF$ --- which is a contradiction with the assumption that $\scf(\uF)>\aleph_0$. It follows that $\uF$ is $k$-Cayley-bounded with respect to $U$, and we conclude that $\uF$ has the Bergman-property. 
	
	``$\Leftarrow$'' Suppose that $\uF$ has the Bergman-property and that $\cf(\uF)>\aleph_0$ but that $\scf(\uF)=\aleph_0$. Let $(U_i)_{i\in\bN}$ be a strong cofinal chain of length $\aleph_0$ for $\uF$. Then in particular we have that 
	$\bigcup_{i\in\bN}U_i=\uF$. However, then we also have 
	$\bigcup_{i\in\bN}
	\langle U_i\rangle_{O_A}= \uF$. Because $\cf(\uF)>\aleph_0$, there exists some $j\in\bN$ such that $\langle U_j\rangle_{O_A}=\uF$. In other words, $U_j$ is a generating set for $\uF$. Since $\uF$ has the Bergman-property, for every $k\in\bN\setminus\{0\}$ there exists an $r\in\bN\setminus\{0\}$ such that $\uF^{(k)}=U_j^{[k,r]}$. Since $(U_i)_{i\in\bN}$ is a strong cofinal chain, from Lemma~\ref{strong2}, it follows that for every $k\in\bN\setminus\{0\}$ there exists some $l\in\bN$ such that $\uF^{(k)}=U_j^{[k,r]}\subseteq U_l$ --- contradiction. Thus $\scf(\uF)>\aleph_0$. 
\end{proof}

\begin{theorem}\label{newscf}
	Let $\cC$ be a strict \Fraisse-class with \Fraisse-limit $\bU$, such that 
	\begin{enumerate}
		\item $\cC$ is closed with respect to finite products,
		\item $\cC$ has the HAP,
		\item the coproduct of countably many copies of $\bU$ in $(\overline{\cC},\to)$ exists,
		\item $\End\bU$ is not finitely generated.
	\end{enumerate}
	Then $\Pol\bU$ has the Bergman property.
\end{theorem}
\begin{proof}
	First we note, that by Theorem~\ref{hapext}, $\bU$ has a universal homogeneous endomorphism $u$. 
	
	Secondly, with the same arguments as in the proofs of Theorem~\ref{clgen} and Proposition~\ref{PolGen}, we see that there exist retractions $r_i:\bU\epito\bU^{i+1}$ ($i\in\bN\setminus\{0\}$) with corresponding sections $\epsilon_{i}:\bU^{i+1}\injto\bU$.
	
	Thirdly, by Theorem~\ref{dol} we have that $\scf(\End\bU)>\aleph_0$.

	We already know from Theorem~\ref{clcf} that $\Pol\bU$ has uncountable cofinality. Hence, by  Proposition~\ref{bergman}, $\bU$ has the Bergman property if and only if it has uncountable strong cofinality. Suppose now, that $\Pol \bU$ has countable strong cofinality. Let $(U_i)_{i\in\bN}$ be a strong cofinal chain for $\Pol\bU$. Define $V_i:=U_i\cap\End(\bU)$.	Clearly, since $\bigcup_{i\in\bN} U_i =\Pol\bU$, it follows that also $\bigcup_{i\in\bN} V_i =\End\bU$. Let now $i\in\bN$ and let $j\in\bN$ be given such that $U_i^{[k,2]}\subseteq U_j$, for all $k\in\bN\setminus\{0\}$. Let $s_1,s_2\in V_i$. Then 
	\[s_1\circ s_2\in V_i^2\subseteq U_i^{[1,2]}\subseteq U_j\cap\End(\bU)= V_j.\]
	It follows that $V_i\cdot V_i\subseteq V_j$. Since $\End\bU$ has uncountable strong cofinality, there has to exist some $j_0\in\bN$ such that $V_{j_0}=\End\bU$. In particular, $\End(\bU)\subseteq U_{j_0}$.

	Let now $k\in\bN\setminus\{0\}$. Then we can find some $j_k>j_0$ such that $U_{j_k}$ contains $\epsilon_k$.  Let $f\in\Pol^{(k+1)}\bU$. Then $f\circ r_{k}\in\End\bU\subseteq U_j\subseteq U_{j_k}$. Since $u$ is a universal endomorphism of $\bU$, there exists an embedding $\iota_f:\bU\injto\bU$ such that $f\circ r_k=u\circ \iota_f$. Define $\hat\iota_f:\bU^{k+1}\injto\bA$ through $\hat\iota_f:=\iota_f\circ\epsilon_k$. Since $\iota_f\in V_{j_k}\subseteq U_{j_k}^{[1,1]}$, and since $\epsilon_k\in U_{j_k}^{[k+1,1]}$, it follows from Lemma~\ref{strong} that $\hat\iota_f\in U_{j_k}^{[k+1,2]}$. Then  $u\circ\hat\iota_f = u\circ\iota_f\circ\epsilon_k = f\circ r_k\circ\epsilon_k=f$. Since $u\in V_{j_k}\subseteq U_{j_k}^{[1,1]}$,  it follows from Lemma~\ref{strong} that $f\in U_{j_k}^{[k+1,3]}$. Hence $U_{j_k}^{[k+1,3]}=\Pol^{(k+1)}\bU$. From Lemma~\ref{strong2}, it follows that there exists some $l\in\bN$ such that $U_{j_k}^{[k+1,3]}\subseteq U_l$ and hence $U_l^{(k)}=\Pol^{(k)}\bU$ --- contradiction. So $\Pol\bU$ has uncountable strong cofinality and hence the Bergman property. 
\end{proof}

\begin{example}\label{exscf}
	The conditions of Theorem~\ref{newscf} are the same as the conditions of Theorem~\ref{clcf}. Hence all the structures from Example~\ref{excf} also have the Bergman property:
	\begin{itemize}
		\item the Rado graph $R$,
		\item the countable generic poset $\mathbb{P}=(P,\le)$,
		\item the countable atomless Boolean algebra $\mathbb{B}$,
		\item the countable universal homogeneous semilattice $\Omega$,
		\item the countable universal homogeneous distributive lattice $\mathbb{D}$,
		\item the infinite-dimensional vector-space $\mathbb{F}^\omega$ for any countable field $\mathbb{F}$.
	\end{itemize}
\end{example}

\section{Some open Problems}
Unfortunately, our results from Section~\ref{genclones} do not apply to the polymorphism clone of $(\bQ,\le)$. Therefore we ask:
\begin{problem}
	Does $\Pol(\bQ,\le)$ have a generating set of bounded arity? What is its relative rank with respect to $\End(\bQ,\le)$, to $\Emb(\bQ,\le)$, or even to $\Aut(\bQ,\le)$?
\end{problem}

In Section~\ref{cofinality} the sufficient condition that we use to show that a given clone $\uF$ has uncountable cofinality relies on the knowledge of the cofinality of some subsemigroup $\uS$ such that $\uF$ has a finite relative rank with respect to $\uS$. For the polymorphism clone of the rational Urysohn space $\mathbb{U}_\bQ$ we know that $\rank(\mathbb{U}_\bQ:\End \mathbb{U}_\bQ)$ is finite.  Unfortunately, we do not know the cofinality of $\End \mathbb{U}_\bQ$ (this is formulated as Problem 4.6 in \cite{Dol11}). So we can not use Proposition~\ref{cfcond} in order to conclude that $\Pol\mathbb{U}_\bQ$ has uncountable cofinality:
\begin{problem}
	Does $\Pol\mathbb{U}_\bQ$ have uncountable cofinality? What is about the polymorphism clone of the rational Urysohn sphere or of $(\bQ,\le)$?
\end{problem}

The similar problem arrises for our results about the Bergman property of clones. The sufficient conditions for the Bergman property of polymorphism clones of homogeneous structures that are given in Theorem~\ref{newscf} contain Dolinka's sufficient conditions for the Bergman property of the endomorphism monoids of homogeneous structures (cf. \cite[Thm4.2]{Dol11}). Unfortunately, these conditions do not apply to the rational Urysohn space, the rational Urysohn-sphere or to $(\bQ,\le)$. Therefore we ask: 
\begin{problem}
	Does the polymorphism clone of the rational Urysohn space have the Bergman property? $(\bQ,\le)$ have the Bergman property? What about the polymorphism clones of the rational Urysohn-sphere or of $(\bQ,\le)$? 
\end{problem}


\end{document}